\documentclass[11pt]{amsart}
\usepackage{amsmath,amssymb,amsthm,array,mathrsfs}

\usepackage[top=1in,bottom=1in,left=1in,right=1in]{geometry}
\usepackage[dvipsnames,svgnames,x11names,hyperref]{xcolor}
\usepackage[pagebackref,colorlinks,citecolor=Mahogany,linkcolor=Mahogany,urlcolor=Mahogany,filecolor=Mahogany]{hyperref}

\usepackage[T1]{fontenc}
\usepackage{lmodern}
\usepackage{microtype}
\usepackage[mathscr]{euscript}

\usepackage{enumerate}
\usepackage{amssymb}
\usepackage{amsmath}
\usepackage{amsthm}
\usepackage{mathtools}

\usepackage{tikz}
\usepackage{tikz-cd}

\usepackage{xparse}
\usepackage{etoolbox}
\usepackage{aliascnt}
\usepackage {hyperref}
\hypersetup{
	colorlinks=true,
	linkcolor=blue,
	filecolor=blue,
	citecolor = blue,
	urlcolor=blue,
}
\usepackage[capitalize,nameinlink,noabbrev]{cleveref}

\AtBeginDocument{%
	\def\MR#1{}
}

\usepackage{enumitem}
\setenumerate[1]{label=(\roman*),}
\setitemize[1]{label=\raisebox{0.25ex}{\boldmath$\cdot$}}

\usepackage[textsize = footnotesize]{todonotes}
\usepackage{comment}

\newcounter{environmentcounteralphabetic}

\numberwithin{environmentcounter}{section}

\newaliascnt{definitioncounteralias}{environmentcounter}

\newaliascnt{remarkcounteralias}{environmentcounter}

\newaliascnt{examplecounteralias}{environmentcounter}

\newaliascnt{constructioncounteralias}{environmentcounter}

\newaliascnt{lemmacounteralias}{environmentcounter}

\newaliascnt{propositioncounteralias}{environmentcounter}

\newaliascnt{corollarycounteralias}{environmentcounter}

\newaliascnt{theoremcounteralias}{environmentcounter}

\newaliascnt{questioncounteralias}{environmentcounter}

\newaliascnt{conjecturecounteralias}{environmentcounter}

\theoremstyle{definition}
\newtheorem{definition}[definitioncounteralias]{Definition}
\newtheorem{construction}[constructioncounteralias]{Construction}
\newtheorem{notation}[definitioncounteralias]{Notation}

\theoremstyle{plain}
\newtheorem{lemma}[lemmacounteralias]{Lemma}
\newtheorem{proposition}[propositioncounteralias]{Proposition}
\newtheorem{corollary}[corollarycounteralias]{Corollary}

\newtheorem{theorem}[theoremcounteralias]{Theorem}
\newtheorem{theoremalphabetic}[environmentcounteralphabetic]{Theorem}
\newtheorem{corollaryalphabetic}[environmentcounteralphabetic]{Corollary}

\theoremstyle{remark}
\newtheorem{remark}[remarkcounteralias]{Remark}
\newtheorem{example}[examplecounteralias]{Example}

\makeatletter
\def\namedlabel#1#2{\begingroup
	#2%
	\def\@currentlabel{#2}%
	\phantomsection\label{#1}\endgroup
}
\makeatother

\newcommand{\R}{\mathbb R}

\newcommand{\Z}{\mathbb Z}

\newcommand{\fS}{\mathfrak S}

\newcommand{\Hom}{\textup{Hom}}

\newcommand{\catfam}{\textbf{CatFam}}
\newcommand{\asm}{\textbf{Asm}}
\newcommand{\wasm}{\textbf{wAsm}}

\newcommand{\aut}{\textup{Aut}}

\newcommand{\id}{\on{id}}
\newcommand{\im}{\on{Im}}

\newcommand{\ob}{\on{ob}}

\newcommand{\Spc}{\textup{Spc}}
\newcommand{\hoSpc}{\textup{hoSpc}}
\newcommand{\hoSp}{\textup{hoSp}}
\newcommand{\Sp}{\textup{Sp}}

\DeclareMathOperator{\Pt}{Pt}
\DeclareMathOperator{\vol}{vol}

\newcommand{\sma}{\wedge}

\newcommand{\op}{\text{op}}

\newcommand{\on}[1]{\operatorname{#1}}

\newcommand{\cal}[1]{\mathcal{#1}}
\newcommand{\bb}[1]{\mathbb{#1}}

\title{Scissors Congruence of the Line and the Regulator}
\author{Ezekiel Lemann}

\begin{document}
\begin{abstract}
We construct explicit generators for the higher scissors congruence $K$-theory of the line. We use this to derive an explicit generating set for the homology of the group of interval exchange transformations. Our proof makes use of an extended version of the regulator (trace) map of Bohmann, Gerhardt, Malkiewich, Merling, and Zahkarevich. 
\end{abstract}

\maketitle
\tableofcontents
\section{Introduction}
 Given two polyhedra $P$ and $Q$, is it possible to cut $P$ into finitely many polyhedral pieces in such a way that these pieces can be reassembled to yield $Q$? If the answer to this question is yes, we say that $P$ and $Q$ are scissors congruent. The classical work of Dehn, Sydler, and Jessen derives algebraic invariants which determine polyhedra in three dimensional Euclidean geometry up to scissors congruence \cite{Deh01,Syd65, Jes72}.  
 \par Later work of Sah and Dupont furthered the use of algebra in understanding scissors congruence problems. In any geometry one can define a scissors congruence group where the elements are polytopes and any two elements are equal if and only if the polytopes are scissors congruent. Sah and Dupont showed that the scissors congruence groups can be expressed in terms of the homology of the isometry group of the geometry, considered as a discrete group \cite{ds90, Dup01}.  This allows for tools from homological algebra to be used in studying scissors congruence problems.

\par  More recently, Zakharevich \cite{zak16}, develops the categorical framework of assemblers which allows for the  study of scissors congruence as a form of $K$-theory. Calculating $K_0$ corresponds to the geometric problem of determining invariants that distinguish polytopes up to scissors congruence. For an assembler $\cal{A}$, the higher $K$-groups can be defined as the homotopy groups of an associated $K$-theory space. One way to construct this space is by group completion. From an assembler $\cal{A}$ one builds a scissors congruence groupoid $\cal{G}(\cal{A})$; whose morphisms encode scissors congruences. The scissors congruence groupoid $\cal{G}(\cal{A})$ also carries symmetric monoidal structure and group completion with respect to this structure yields the $K$-theory space $K(\cal{A}) = \Omega B^{\sqcup} (B \cal{G}(\cal{A}))$. 
 \par A natural project is then to understand the higher $K$-groups.  
 By $\cal{E}^1_T$, we denote the assembler which encodes the data of 1-dimensional translational scissors congruence. An automorphism in the scissors congruence groupoid $\cal{G}(\cal{E}_T^1)$ is an interval exchange transformation. 
 In fact, the automorphism group of an interval $\aut_{\cal{G}(\cal{E}_T^1)} [0,r]$, considered as an object in the scissors congruence groupoid $\cal{G}(\cal{E}_T^1)$,  is isomorphic to  $IET$, the group of interval exchange transformations \cite{MR609891}. 
 \par Two intervals are scissors congruent if and only if they have the same length, therefore $K_0(\cal{E}^T_1) \cong \bb{R}$. However, 1-dimensional geometry has non-trivial higher $K$-groups. Our main result constructs generating sets for these groups.  To this end, we construct sub-complexes of the classifying space $B\cal{G}(\cal{E}_T^1)$ which correspond to the generators of the $K$-groups $K_n(\cal{E}_T^1) \cong \pi_n \Omega B (B^{\amalg} \cal{G}(\cal{E}_T^1))$. 
 
 \par A special case of the main theorem of \cite{Mal24} shows that 
 \begin{theorem}\cite{Mal24}
     \[K_n (\cal{E}_T^1) \cong H_{n+1} (\bb{R}; \bb{Z}) \cong \bb{R}^{\sma_{\bb{Q}} (n+1)}. \]
 \end{theorem}
 Using the isomorphism of the previous theorem, we establish the following results. 
 \begin{theoremalphabetic}\label{Kgen}
     Let $\{\Phi_i\}_{i=0} ^n=X \subset \bb{R}$. The associated generator $\Phi_0 \sma \dots \sma \Phi_n \in K_n (\cal{E}^1_T)$ is realized by a sub-complex $C_X= B(\langle \rho_{i} | 1 \leq i \leq n \rangle )$, where $\rho_{i}$ is the interval exchange transformation corresponding to rotating a circle with circumference $\Phi_0 + \cdots + \Phi_n$ clockwise by $\Phi_i$. 
 \end{theoremalphabetic}
In other words, the generators arise from group homomorphisms $\bb{Z}^n \rightarrow IET$ where $\bb{Z}^n$ is identified with a subgroup generated by $n$ rationally independent circle rotations. 
  By the theory developed in \cite{klmms24} we obtain the following corollary. 
 \begin{corollaryalphabetic}\label{cor:IETgen}
     Sub-complexes of the type $C_X$ generate $H_* (IET)$ as a ring. 
  \end{corollaryalphabetic}
  The calculations involved in Theorems \ref{Kgen} and \ref{cor:IETgen} depend on an alternative description of the regulator, which is called the trace in \cite{bgmmz}. 
  \begin{theoremalphabetic}\label{theorem:regulatorextension}
      The regulator of \cite{bgmmz} has an equivalent description using the scissors congruence groupoid $\cal{G}(\cal{A}_{hG})$. 
  \end{theoremalphabetic}
  As in \cite[Example 7.10]{bgmmz}, this allows us to recover a rational equivalence in the case of classical scissors congruence. 
\begin{corollaryalphabetic}\label{Cor:RatEquiv}
      Let $\cal{X}$ be the assembler encoding $n$-dimensional spherical, hyperbolic, or Euclidean scissors congruence. Let $G$ be a subgroup of the isometry group. There is a rational equivalence 
      \[ K_* (\cal{X}_{hG}) \rightarrow H_* (G; \Pt(\cal{X}))\]
      induced by the universal regulator. 
  \end{corollaryalphabetic}
   \par Several recent developments in scissors congruence $K$-theory play an important role in this work.  We make use of the regulator (trace) defined by Bohmann, Gerhardt, Malkiewich, Merling, and Zakharevich \cite{bgmmz} and give another construction of the regulator using the group completion model for $K$-theory given in \cite{klmms24}.
The Tits building model for scissors congruence $K$-theory  given by Malkiewich \cite{Mal24} is used in conjunction with  results of Kupers, the author,  Malkiewich, Miller, and Sroka which relates $K$-theory to the scissors automorphism group of a polytope \cite{klmms24}.
\subsection{Notations}
We use the notation $\Spc$ for the category of spaces and $\Sp$ for the category of spectra, decorating with a $*$ if base points are used. We write $\fS_n$ for the permutation group of the set $\{1, \dots , n\}$. For $\sigma \in \fS_n$ we write $(-1)^\sigma$ for the signature of the permutation $\sigma$. When convenient we use the shorthand $[v \mid w ] = [v] \otimes [w]$. We write $I(E^n)$ for the isometry group of Euclidean space $E^n$. The subgroup of $I(E^n)$ consisting of translations is written $T_n$. The $n$-dimensional torus is $T^n$. 
\subsection{Acknowledgments}
  The author thanks Robin Sroka for valuable discussions and Cary Malkiewich for insight, patience, and support. This project was supported by an NSF grant for the Focused Research Group project
Trace Methods and Applications for Cut-and-Paste K-theory DMS-205292.
This paper represents a part of the author's Ph.D. thesis at Binghamton University.
\subsection{Organization}
Section 2 gives background on group completion $K$-theory in the context of weak assemblers. 
Section 3 redevelops the theory of the regulator using the group completion model of $K$-theory, ending with \cref{theorem:regulatorextension}  (\cref{theorem:regulator_extension_proved}) and \cref{Cor:RatEquiv} (\cref{cor:Cor_D_u_reg}). 
Section 4 specializes to the case of 1-dimensional geometry, giving a derivation of the volume regulator. Section 5 concerns constructing generators, giving proofs of  \cref{Kgen} (\cref{complexrecipe} and \cref{prop:torusrepresentsgenerator}) and \cref{cor:IETgen} (\cref{cor:IETgen_proved}). 
Section 6 concludes the paper with a proposition showing how the generators of $H_*(IET)$ give rise to generators $H_*(Rec_n)$, the homology of the group of rectangle exchange transformations \cite{cornulier2022groupsrectangleexchangetransformations}.

\section{A presentation of $\cal{G}(\cal{A})$.} 
In this section we recall the  the category of fractions construction from \cite[Section 2.3]{klmms24} in the slightly more general setting of weak assemblers. 
\subsection{Weak assemblers}\label{sec:def-assemblers}
 The definition of an assembler \cite[Def.2.4]{zak16} involves a Grothendieck topology, but for our purposes this isn't necessary so we will remove this assumption. We need the structure of finite covering families and some additional axioms, which are also used in the definition of an assembler, that allow for the construction of $K$-theory by group completion. 
\begin{definition}\cite[Definition 2.1]{bgmmz} Given a category $\cal{C}$, a \emph{multimorphism} in $\cal{C}$ is a finite set of morphisms in $\cal{C}$ which share a codomain: 
\[\{\phi_i:  a_i \rightarrow b\}_{i \in I}. \] The set $I$ is finite and the case $I = \emptyset$ is permitted. 
    A \emph{covering family structure} on $\cal{C}$ is a collection of multimorphisms of $\cal{C}$, called \emph{covering families}, which satisfy the following axioms: 
    \begin{enumerate}[label=\bf{(C\arabic*)}]
        \item For every object $a \in \cal{C}$, $\{1_a : a \rightarrow a\} $ is a covering family.
        \item Covering families are composable. Given covering families 
        \[\{\phi_i : b_i \rightarrow c\}_{i \in I}, \, \{\psi_{ji}: a_{ji} \rightarrow b_i\}_{j  \in J_i}\] 
       then the composition 
        \[\{\psi_{ji} \circ \phi_i : a_{ji} \rightarrow c\}_{i \in I, j \in J_i}\]
        is a covering family. 
    \end{enumerate}
  A small category $\cal{C}$ with distinguished basepoint $ *\in \cal{C}$ and covering family structure is a \emph{category with covering families} if the following hold: 
  \begin{enumerate}
      \item $\hom(*,*)=\{1_*\}$ and $\hom(A,*)=\emptyset$ for $A \neq *$.
      \item For every finite set $I$, the set $\{ * \rightarrow * \}_{i \in I}$ is a covering family. 
  \end{enumerate}
  Given a category with covering families we use the following terminology.
  \begin{itemize}
      \item We say that the covering family 
      \[ \{\psi_j: b_j \rightarrow a\}_{j \in J}\]
      \emph{refines} the covering family 
      \[\{\phi_i : a_i \rightarrow a\}_{i \in I}\]
      if there exists a function $f: I \rightarrow J$ and a covering family
      \[ \{\eta_j: b_j \rightarrow a_i\}_{j \in f^{-1}(i)}\]
      for every $i \in I$, such that $\phi_{f(j)}\circ\eta_j=\psi_j$. \\
      \item We say that a multimorphism \[ \{a_i \rightarrow a\}_{i \in I}\] is \emph{disjoint} if every pullback $a_i \times_{a} a_j, j \neq i$, exists and is initial. 
  \end{itemize}
       
 Given categories with covering families $\cal{C}$ and $\cal{D}$, a functor $F:\cal{C} \rightarrow \cal{D}$ which preserves distinguished base points and covering family structure is a \emph{morphism of categories with covering families}. 
  \end{definition}

\begin{definition}\label{def:wasm}
 A category  $\cal{A}$ equipped with a collection of finite disjoint covers $\{a_i \rightarrow a\}_{i \in I}$ is called a  \emph{weak assembler} if the following axioms are satisfied. 
\begin{description}
\item[\namedlabel{enum:wasm-id}{(C)}] For evey object $a$ in $\cal{A}$ the identity $\{a = a\}$ is a covering family. Covering families are closed under composition. 
\item[\namedlabel{enum:wasm-initial}{(I)}] $\cal{A}$ has an initial object $\emptyset$. The empty family is a cover of $\varnothing$. Additionally, covers are closed under adding or taking away a copy of $\varnothing$. 
\item[\namedlabel{enum:wasm-refinement}{(R)}] Any two covers of an object $P$ of $\cal{A}$ have a common refinement.
\item[\namedlabel{enum:wasm-mono}{(M)}] Every morphism of $\cal{A}$ is a monomorphism.
\end{description}
\end{definition}
\begin{remark}\label{remark:ForgettingwAsmStructure}
    Forgetting some of the axioms defining a weak assembler yields a category with covering structure. 
\end{remark}

\begin{definition}\label{defsc}
    Two objects $a,b$ of a weak assembler $\cal{A}$ are \emph{scissors congruent} if there are covers $\{a_i \rightarrow a\}_{i \in I}$ and $\{b_i \rightarrow b\}$ such that $a_i \cong b_i$ for every $i \in I$.  
\end{definition}
\begin{definition}
    We define $\wasm$ to be \emph{the category of weak assemblers}. The objects are weak assembler and the morphisms are functors which preserve covering structure and the initial object. 
\end{definition}
\subsection{The category of covers $\cal{W}(\cal{A})$.}
From a weak assembler $\cal{A}$ we build a category $\cal{W}(\cal{A})$ called the category of covers. Morphisms in this category encode scissors congruence.
\begin{definition}
    Given a weak assembler $\cal{A}$ the \emph{category of covers} $\cal{W}(\cal{A})$ is defined as follows: 
    \begin{itemize}
    \item The objects of $\cal{W}(\cal{A})$ are collections of objects from $\cal{A}$ indexed by a finite set. We use the notation $\{P_i \}_{i \in I}$ where $I$ is a finite set and $P_i \in \cal{A}$.
    \item Morphisms $f:\{P_i\}_{i \in I} \rightarrow \{Q_j\}_{j \in J}$ are given by a function $f:I \rightarrow J$ and morphisms from $\cal{A}$,  $f_i: P_i \rightarrow Q_{f(i)}$ such that $\{f_i : P_i \rightarrow Q_j\}_{i \in f^{-1}(j)} $ is a cover. We record morphisms as tuples: $(f, \{f_i\}_{i \in I})$. 
    \item Composition is given by $(f, \{f_i\}_{i \in I}) \circ (g, \{g_j\}_{j \in J})=(fg, \{f_{g(j)} \circ g_j\}_{j \in J})$. 
    \item $\cal{W}(\cal{A})$ has symmetric monoidal structure given by \[ (I, \{P_i\}_{i \in I}) \sqcup (J, \{Q_j\}_{j \in J}) = (I \sqcup J, \{P_i\}_{i \in I} \sqcup \{Q_j\}_{j \in J}).\]
    \end{itemize}
\end{definition}
\begin{remark}
     Singletons  $\{P\}$ and $\{Q\}$ are connected by a zigzag of morphisms in $\cal{W}(\cal{A})$ if and only if $P$ and $Q$ are scissors congruent. 
\end{remark}
\begin{definition}
    The scissors congruence $K$-theory spectrum, $K(\cal{A})$, of an assembler $\cal{A}$ is the symmetric spectrum obtained by applying the Segal $\Gamma$ construction to $\cal{W}(\cal{A})$. 
    \end{definition}
\subsection{The groupoid of scissors congruences $\cal{G}(\cal{A})$}
\par It is known that the localization $\cal{G}(\cal{A})=\cal{W}(\cal{A})[\cal{W}(\cal{A})^{-1}]$, which we call the \emph{scissors congruence groupoid}, satisfies the conditions for a category of fractions where morphisms are represented by spans \cite[Def.2.16]{klmms24}. 
The objects of $\cal{G}(\cal{A})$ are the same as those of $\cal{W}(\cal{A})$, but morphisms are equivalence classes of spans. We will use the presentation of $\cal{G}(\cal{A})$ to give an alternate proof  that the localization functor induces a weak equivalence $B(\cal{W}(\cal{A})) \simeq B(\cal{G}(\cal{A}))$ \cite[Theorem 4.1]{klmms24}.

\par We begin by recalling the category of fractions construction, following Gabriel and Zisman \cite[Ch.I]{gabriel2012calculus}.    

\begin{definition}Given  a subcategory of weak equivalences $W \subseteq \cal{C}$, the localization $\cal{C}[W^{-1}]$ admits a calculus of fractions if the following conditions hold. 

\begin{enumerate}[label= \bf{({CF}\arabic*)}]
\item Ore Condition. Given $f: x \rightarrow z \in W$ and and $g: y \rightarrow z \in \cal{C}$ there exist $h: w \rightarrow x \in \cal{C}$ and $i : w \rightarrow y \in W$ such that the square commutes: 
\[ \begin{tikzcd}
w \ar[r] \ar[d]& x \ar[d]  \\
y \ar[r] & z.
\end{tikzcd}\] 

\item Given \[ \begin{tikzcd} x \ar[r, "f", shift left=.75ex] \ar[r, "g"', shift right = .75ex] & y \ar[r, "h"] & z \end{tikzcd}\] such that $hf=hg$ there exists an object $w$ and morphism 
\[ \begin{tikzcd} w \ar[r, "i", dotted ]& x \ar[r, "f", shift left=.75ex] \ar[r, "g"', shift right = .75ex] & y \ar[r, "h"] & z \end{tikzcd}\] 
such that $fi=gi$.
\end{enumerate}

\par Under these conditions, the morphism sets in $\cal{C}[W^{-1}]$ can be described as spans where the left leg is in $W$, under an equivalence relation: 
\[ \cal{C}[W^{-1}](a,b) = \{a \xleftarrow{f} c \xrightarrow{g} b : f \in W  \}/ \sim .\]

Where we define $a \xleftarrow{f} c \xrightarrow{g} b \sim a \xleftarrow{h} c' \xrightarrow{i} b$ if there exists an object $x$ and morphisms $u,v$ such that the following diagram commutes: 

\[ \begin{tikzcd}
& b \ar[dl , "f"'] \ar[dr, "g"]& \\
a & \ar[u, "u"']x \ar[d, "v"]& c \\
&b' \ar[ul, "h"]\ar[ur, "i"']&
\end{tikzcd}\]
and $fu = hv \in W$.
\par Given spans $(b \leftarrow b' \rightarrow c) $ and $ (a \leftarrow a' \rightarrow )b$ , construct the diagram
\[ \begin{tikzcd}
a  &a' \ar[l] \ar[r] &b \\
& z \ar[u] \ar[r] & b' \ar[u] \ar[d]\\
& & c.
\end{tikzcd}\] 

where $z$ exists by (CF1). Then the composition is defined by  
\[ (b \leftarrow b' \rightarrow c) \circ (a \leftarrow a' \rightarrow b ) = (a \leftarrow a' \leftarrow z \rightarrow b' \rightarrow c ).\]
\par Using the calculus of fractions construction the localization functor $\iota_W:\cal{C} \rightarrow \cal{C}[W^{-1}]$ is defined by 
\[ x \xrightarrow{f} y \mapsto x \xleftarrow{1_x} x \xrightarrow{f} y. \]
\end{definition}

\par In the context of this paper, our interest in categories of  fractions comes from the following. 
\begin{proposition}[{\cite[Lemma 2.15]{klmms24}}]\label{assemblerfractions}
For $\cal{A}$ a weak assembler, $\cal{G}(\cal{A})$ has a representation as a category of fractions . 
\end{proposition}

This gives a model for $\cal{G}(\cal{A})$ that can be useful for computations. We will use the category of fractions description to show that $B(\cal{W}(\cal{A})) \simeq B(\cal{G}(\cal{A}))$.

\begin{definition}\label{definition:Filtered_cat}
Recall that a category $\cal{C}$ is \emph{filtered} if the following conditions are satisfied: 
\begin{enumerate}[label = \bf{(F\arabic*)}]
\item  For every $x,y \in \cal{C}$ there exists some $z \in \cal{C}$ such that there exist morphisms $x \rightarrow z \leftarrow y$. 
\item For every pair \begin{tikzcd} x \ar[r , shift left = .75ex, "f"]\ar[r , shift right = .75ex, "g"'] & y  \end{tikzcd} there exists an object $z$ with morphism $h: y \rightarrow z$ such that $hf=hg$. 
\end{enumerate}  
\end{definition} 

\begin{lemma}[{\cite[Def. 2.6.13]{weibel1994introduction}}]\label{lemma:Bcontractible} Filtered categories have contractible classifying spaces. 
\end{lemma}

\begin{proposition}\label{prop:locWE}
Let $\cal{C}$ be a category where every cospan can be completed to a commutative square and all morphisms are monic. Then the localization functor induces a weak equivalence $B(\cal{C}) \simeq B(\cal{C}[\cal{C}^{-1}])$.
\end{proposition}

\begin{proof}
\par Let $\iota: \cal{C} \rightarrow \cal{C}[\cal{C}^{-1}]$ be the localization functor.  Now we will show that $(\iota/c)^\op$ is a filtered category, or rather that $\iota/c$ is co-filtered.  Given two objects $x \leftarrow x' \rightarrow c$ and $y \leftarrow y' \rightarrow c$ in $\iota/c$ we can complete the cospan $x' \rightarrow c  \leftarrow y'$  to a commutative square with morphisms in $\cal{C}$ yielding a commutative diagram:
\[ \begin{tikzcd} 
x & z \ar[dl , dashed] \ar[dr , dashed] & y \\
x' \ar[u] \ar[dr] &  & y' \ar[u] \ar[dl]\\
& c. & 
\end{tikzcd}\]

By composition we have morphisms $z \rightarrow x$ and $z \rightarrow y$ from $\cal{C}$. Therefore the object $z \rightarrow c$ of $\iota/c$ maps to both the given objects $x \leftarrow x' \rightarrow c$ and $y \leftarrow y'\rightarrow c$, proving that (F1) is satisfied in $(\iota/c)^\op$. 

\par Given parallel morphisms in $\iota/c$: 
\[\begin{tikzcd}
x \ar[rr , shift right = .75, "g"'] \ar[rr, shift left = .75 ex,"f"] &&y  \\
x'\ar[u] \ar[dr] && y' \ar[u] \ar[dl]\\
& c.&
\end{tikzcd}\]

commutativity and the fact that all morphisms are isomorphisms in $\cal{C}[\cal{C}^{-1}]$ implies that $f=g$ in $\cal{C}[\cal{C}^{-1}]$. Therefore there is a diagram of the form 
\[\begin{tikzcd} 
& \ar[dl, "1_x"'] x  \ar[dr, "f"] &\\
x  & z \ar[u, "\phi"'] \ar[d, "\psi"] & y  \\
& x. \ar[ul, "1_x" ] \ar[ur , "g"']& 
\end{tikzcd} \]

which implies that $\phi=\psi$ has the property $f\phi =g \phi$. Therefore the zigzag $z \xrightarrow{\phi}  x \leftarrow x' \rightarrow c$ shows that $\iota/c$ satisfies (F2).  
Now that we have shown $\iota/c$ is contractible for all $c\in C$, we conclude by Quillen's Theorem A \cite[Ch.IV,3.7]{weibel2013k} that $\iota $ induces an equivalence on classifying spaces: $B\iota:B(\cal{C}) \simeq B\cal{C}[\cal{C}^{-1}])$. 
\end{proof}
\begin{corollary}
    If $\cal{A}$ is an assembler then $B(\cal{W}(\cal{A})) \simeq B(\cal{G}(\cal{A}))$.
\end{corollary}
\begin{proof} Let $\cal{A}$ be an assembler. And let $\cal{W}^\circ (\cal{A})$ be the version of the category of covers where we do not allow initial objects in the tuples. 
By \cite[Prop.2.11]{zak16}, we have that any cospan in $\cal{W}^\circ(\cal{A})$ can be completed to a square and every morphism is monic. Therefore $\cal{W}(\cal{A})[\cal{W}(\cal{A})^{-1}]$ satisfies the conditions of \cref{prop:locWE} above. By the adjunction used in \cite[Lemma 2.15]{klmms24} the result may be extended to the version of $\cal{G}(\cal{A})$ where initial objects are permitted. 
\end{proof}

\subsection{Homotopy orbits and the Grothendieck construction for weak assemblers}
We will recall some results and terminology from \cite{bgmmz}. Some results from this section hold for categories with covering families, a generalization of assemblers, but for our purposes we consider the case of an assembler.

\begin{definition}
    Let $G$ be a group, considered a category with one object and only invertible morphisms. A $G$-category with covering family structure is a functor $F: G \rightarrow \catfam$. We use $*_G$ for the object of $G$. 
\end{definition}
\begin{definition}\label{def:grothwasm}
	Considering a group $G$ as a category with one object, a functor  
    \[ F: G \rightarrow \aut_{w\asm}(\cal{A})\]
    makes $\cal{A}$ a $G$-category with covering family structure. Unless multiple actions are present, we will suppress $F$ from the notation, writing $\Phi A=F(\Phi)A$ and $\Phi \psi=F(\Phi)\psi$ where $A$ is an object of $\cal{A}$ and $\psi$ a morphism of $\cal{A}$. 
	The category $\cal{A}_{hG}=G \widetilde{\int} F$ is a modification of the classical Grothendieck construction, $G\int F$ which is suitable for the pointed setting defined by: 
 \begin{align*}
     \ob \cal{A}_{hG} &= \ob F(*_G) = \ob \cal{A}\\
     \cal{A}_{hG} (a,b) &= \bigsqcup_{\Phi \in G} \cal{A}(\Phi a, b) 
\end{align*}
Composition of two arrows
\[ a \xleftarrow{(\phi,\Phi)} b \xleftarrow{(\psi,\Psi)} c\]
is defined as
\[a \xleftarrow{(\phi \Phi(\psi), \Phi\Psi)} c.\]

 In the special case when $A=*$ , we define a morphism $* \rightarrow B$ to be the same as a morphism in $\cal{C}$. If $\psi: * \rightarrow B$ and $(\phi, \Phi): B \rightarrow C$ then the composition is given by $(\phi, \Phi) \circ \psi = \phi \circ \Phi (\psi)$.
\par Covering structure is defined by declaring 
\[ \{  a  \leftarrow a_i: (\phi_i, \Phi_i)\}_{i \in I}\]  to be a cover whenever 
\[\{  a \leftarrow \Phi_ia_i: \phi_i\}_{i \in I} \] 
is a cover in $\cal{A}$.   
  \end{definition} 

\begin{remark}
    We will use the convention that the lowercase Greek letters are morphisms of the category $\cal{A}$ and the uppercase Greek letters are group elements. 
\end{remark}

\begin{proposition}\label{prop:htpyorbitwassm}
    If $\cal{A}$ is a weak assembler and $F:G \rightarrow \aut_{\wasm}(\cal{A})$ defines an action by isomorphisms of weak assemblers, then $\cal{A}_{hG}$ is a weak assembler.
\end{proposition}
\begin{proof}
  We verify axioms (R) and (M) for $\cal{A}_{hG}$. 
  \par For (M) we need to show that all morphisms in $\cal{A}_{hG}$ are monic. Suppose that the diagram 
  \[
  \begin{tikzcd}
      a & b \ar[l, "(\phi{,} \Phi)"'] & x\ar[l, "(\xi{,}\Xi)", shift left=.75ex] \ar[l, "(\omega{,}\Omega)"', shift right = .75ex] 
  \end{tikzcd}
  \]
  commutes in $\cal{A}_{hG}$ so that we have 
  \begin{align*}
      (\phi, \Phi) \circ (\omega , \Omega) &=  (\phi, \Phi) \circ (\xi, \Xi)\\
      (\phi \Phi(\omega), \Phi\Omega) &= (\phi \Phi(\xi), \Phi\Xi).
  \end{align*}
  Therefore $\phi \Phi(\omega)=\phi \Phi(\xi)$ in $\cal{A}$, which yields $\Phi(\omega)=\Phi(\xi)$ since all morphisms in $\cal{A}$ are monic. Since $G$ is a group action we have $\omega = \xi$ in $\cal{A}$. By the definition of morphism in $\cal{A}_{hG}$ we have $\Phi \Omega= \Phi \Xi$ which implies $\Omega = \Xi$ since $\Phi$ is invertible. Therefore $(\omega, \Omega) = (\xi, \Xi)$, proving that every morphism in $\cal{A}_{hG}$ is monic. 
  \par For (R) take two covers in $\cal{A}_{hG}$ 
\begin{align*}
    \{a \leftarrow b_i : &(\phi_i, \Phi_i) \}_{i \in I}\\
    \{a \leftarrow c_j : &(\psi_j, \Psi_j) \}_{j \in J}.
\end{align*}
By definition there are corresponding covers 
\begin{align*}
    \{a \leftarrow \Phi_ib_i : &\phi_i \}_{i \in I}\\
    \{a \leftarrow \Psi_j c_j : &\psi_j \}_{j \in J}.
\end{align*}
in $\cal{A}$ which have a common subdivision 
\[
\{a \leftarrow d_k : \eta_k \}_{k \in K}
\]
since $\cal{A}$ is a weak assembler.
The common subdivision includes the data of functions $f: K \rightarrow I $ and $g: K \rightarrow J$ such that there are covers
\begin{align*}
    \{\Phi_ib_i &\xleftarrow{\lambda_k} d_k\}_{k \in f^{-1}(i)}\\
    \{\Psi_j c_j &\xleftarrow{\gamma_k} d_k\}_{k \in g^{-1}(j)}
\end{align*}
for every $i \in I$ and every $j \in J$ such that 
\[
    \psi_{g(k)} \gamma_k=\eta_k = \phi_{f(k)} \lambda_k
\]
for every $k \in K$. 
\par We claim that 
\[ 
\{a \xleftarrow{(\eta_k,1_G)} d_k\}_{k \in K}
\]
is a common subdivision of the given covers from $\cal{A}_{hG}.$
Since $G$ acts by functors which preserve covering structure there are covers,
\begin{align*}
    \{b_i &\xleftarrow{\Phi^{-1}_{f(k)}(\lambda_k)} \Phi^{-1}_{f(k)}d_k\}_{k \in f^{-1}(i)}\\
    \{ c_j &\xleftarrow{\Psi_{g(k)}^{-1}(\gamma_k)} \Psi^{-1}_{g(k)}d_k\}_{k \in g^{-1}(j)}
\end{align*}
for every $i \in I$ and every $j \in J$.
Verifying that  
\[
    (\phi_{f(k)}, \Phi_{f(k)}) \circ (\Phi^{-1}_{f(k)} (\lambda_k), \Phi^{-1}_{f(k)})=(\eta_k, 1_G) = (\psi_{g(k)}, \Psi_{g(k)}) \circ (\Psi^{-1}_{g(k)} (\gamma_k), \Psi_{g(k)}^{-1})
\]
shows that $(\eta_k, 1_G)_{k \in K}$ defines a common subdivision.
\end{proof}

\begin{definition}
Since $\cal{A}_{hG}$ is a weak assembler we may take the associated category of covers $\cal{W}_{hG}$. We will make some conventions when working with $\cal{W}_{hG}$. For a morphism $(\phi_i, \Phi_i): \{a_i\}_{i \in I} \rightarrow \{b_j\}_{j \in J}$ we will write the associated set map $\underline{\phi}: I \rightarrow J$. 

\par In $\cal{W}(\cal{A}_{hG})$ we distinguish two special classes of morphisms. 
\begin{itemize}
\item A \emph{move} is a morphism of the form $(\phi_i, \Phi_i): \{a_i\}_{i \in I } \rightarrowtail \{b_i\}_{i \in I}$ such that for each component $\phi_i: a_i \rightarrow b_{\underline{\phi}(i)}$ we have $\phi_i = 1_{a_i}$.\\ 
 \item  A \emph{covering sub-map} is a morphism $(\phi_i, \Phi_i):\{a_i\}_{i \in I} \dashrightarrow \{b_j\}_{j \in J}$ such that on components $(\phi_i, \Phi_i): P_i \rightarrow Q_{\underline{\phi}(i)}$ we have $\Phi_i=1_G$. 
 \end{itemize}
 \end{definition}
 \begin{remark}
     The intuition is that moves use just group elements to rearrange a collection of pieces without changing their indices in $I$, whereas a covering sub-map sticks a collection of pieces together. There is a functor $\iota_{hG}: \cal{A} \rightarrow \cal{A}_{hG}$  defined  by 
     \[
     \iota_{hG}(a \xrightarrow{\phi} b)=a \xrightarrow{(\phi, 1_G)}b. \]
     A covering sub-map is just a morphism in the image of 
     \[ \cal{W}(\iota_{hG}) : \cal{W}(\cal{A}) \rightarrow \cal{W}(\cal{A}_{hG}).\] 
 \end{remark}
 
\begin{example}[{\cite[Example 4.5]{bgmmz}}]\label{htpyorbiteuc}
    Let $\cal{E}_e^n$ be the no-moving assembler of euclidean space $E^n$.  The objects are polytopes in $E^n$, the morphisms are inclusions $a \subseteq b$ where $a$ and $b$ are polytopes. A finite collection of morphisms 
    $\{\phi_i: a_i \rightarrow a\}_{i \in I}$ is a cover if $a=\bigcup a_i$ and the $a_i$ are pairwise interior disjoint. 
    \par  Let $G \leq \aut(E^n)$ be a subgroup of Euclidean isometries. 
    There is a functor $F:G \rightarrow \aut_{\wasm}(\cal{E}^{n}_{e})$ where $G$ acts by isometries. From this data we form the homotopy orbit category $(\cal{E}^n_{e})_{hG}=G \widetilde{\int}  F$. We denote this category by $\cal{E}^n_G$ as in \cite[Notation 5.2]{klmms24}. 
\end{example}
The morphisms in $\cal{A}_{hG}$ may be factored into a move followed by a covering sub-map. Morphisms in $\cal{A}_{hG}$ contain data from the $G$-action and from the data of the morphisms in $\cal{A}$: these features of a morphism can be separated in $\cal{A}_{hG}$.
 \begin{definition}[{\cite[Definition 5.3]{bgmmz}}]\label{doublecatstructure}
  There is a double category $\cal{W}(\cal{A}_{hG})$ with 
 \begin{itemize}
     \item $\ob \cal{W}^\square (\cal{A}_{hG})= \ob \cal{W} (\cal{A}_{hG})$.
     \item Horizontal morphisms are moves.
     \item Vertical morphisms are covering sub-maps.
     \item 2-cells are commuting squares in $\cal{W}(\cal{A}_{hG})$.
 \end{itemize}
\end{definition}
The following lemmas detail properties of $\cal{W}^\square (\cal{A}_{hG})$.
\begin{lemma}[{\cite[Lemma 5.3]{bgmmz}}] \label{factoring} Every morphism in $\cal{W}(\cal{A}_{hG})$ factors up to unique isomorphism as a move followed by a covering sub-map 
    \[
    \begin{tikzcd}
        \bullet && \bullet \ar[ll] \ar[dl, tail]\\
        &\bullet. \ar[ul, dotted]&
    \end{tikzcd}
    \]
\end{lemma}
\begin{lemma}[{\cite[Lemma5.5]{bgmmz}}]\label{completingsquares}
    Each diagram in $\cal{W}^{\square} (\cal{A}_{hG})$ of the form 
    \[ 
    \begin{tikzcd}
        & \{a_i\}_{i \in I} \ar[d, dotted]\\
        \{b_j\}_{j \in J} & \ar[l, tail] \{c_k\}_{k \in K}
    \end{tikzcd}
    \text{ or }
    \begin{tikzcd}
        \{b_j\}_{j \in J}& \ar[l, tail] \{a_i\}_{i \in I} \ar[d, dotted]\\
        &  \{c_k\}_{k \in K}
    \end{tikzcd}
    \] completes uniquely (up to unique isomorphism) to a 2-cell in $\cal{W}(\cal{A}_{hG})$.
\end{lemma}
\begin{lemma}\label{nomovesquare}
Any diagram of the form 
\[ \begin{tikzcd}
\{a_i\}_{i \in I} \ar[d, dotted]& \\
\{b_j\}_{j \in J}& \{c_k\}_{k \in K} \ar[l, dotted]\\
\end{tikzcd}\]
can be completed to a commutative square of the form
\[ \begin{tikzcd}
\{a_i\}_{i \in I} \ar[d, dotted]& \{d_l\}_{l \in L} \ar[d, dotted] \ar[l, dotted]\\
\{b_j\}_{j \in J}& \{c_k\}_{k \in K}. \ar[l, dotted]\\
\end{tikzcd}\]
\end{lemma}
\begin{proof}
 The subcategory of $\cal{W}(\cal{A}_{hG})$ consisting of covering sub-maps is isomorphic to $\cal{W}(\cal{A})$. The desired commuting square exists by axiom (R) in $\cal{A}$. 
\end{proof}
 
\begin{definition}
We call a span of the form 
\[\begin{tikzcd}
   \bullet & \ar[l, dotted]\bullet & \ar[l, tail]  \ar[r, dotted] \bullet & \bullet 
\end{tikzcd}\]
a \emph{DMC-span}. Since an inverted covering sub-map can be thought of as a dissection, we use DMC for  "dissection, move, covering sub-map." 
\end{definition}
\begin{notation}
    We follow the decluttering convention of \cite[Notation 5.7]{bgmmz}. We write $f_{i|j}$ in place of $(f_i)_j$ for the $j$-th component of the morphism $f_i$. 
\end{notation}
 \begin{proposition}\label{prop:rep}
 Every morphism 
\[ \begin{tikzcd}
\{a_{0|i}\}_{i \in I_0}  & \{b_{0|j}\}_{j\in J_0}   \ar[l, "(\phi_j{,}  \Phi_j)"'] \ar[r, "(\psi_j {,} \Psi_j)"] & \{a_{1|i}\}_{i \in I_1}
\end{tikzcd} \] 
in $\cal{G}(\cal{A}) $ can be represented by an DMC-span.  
\end{proposition}

\begin{proof}
By applying the factorization into a move followed by a covering sub-map to the given span we obtain the commutative diagram: 

\[ \begin{tikzcd}
\{a_{0|i}\}_{i \in I_0}  && \{b_{0|j}\}_{j\in J_0} \ar[dl, tail, "\{\Phi_j\}"'] \ar[dr, tail, "\{ \Psi_j\}"]  \ar[ll, "(\phi_j{,}  \Phi_j)"'] \ar[rr, "(\psi_j {,} \Psi_j)"] && \{a_{1|i}\}_{i \in I_1}  \\
&\{\phi_j b_{0|j}\}_{j \in J_0}\ar[ul, dotted, "\{\phi_j\}"]&& \{\Psi_{j}b_{0|j}\}_{j \in J_0}. \ar[ur, dotted, "\{\psi_j\}"'] \ar[ll, tail, "\{\Phi_j \Psi^{-1}_j\}"]&
\end{tikzcd}\]
The move at the bottom exists by the fact that, for a group action, moves are invertible and the inverse of a move is a move. We have produced the desired equivalence. Since we are working in the groupoid $\cal{G}(\cal{A})$ every morphinism is invertible so the composite along the bottom of the diagram represents the same morphism as the composite along the top. 
\end{proof}
\begin{definition}\label{dmcrefinement}
    We say that the DMC-span 
\[\begin{tikzcd}
    \bullet & \bullet \ar[l, dotted, red] & \bullet \ar[l, tail, red] \ar[r, dotted, red]& \bullet
\end{tikzcd}\]
is a \emph{refinement} of the DMC-span 
\[\begin{tikzcd}
    \bullet & \bullet \ar[l, dotted, blue] & \bullet \ar[l, tail, blue] \ar[r, dotted, blue]& \bullet
\end{tikzcd}\]
if there exists a commuting diagram 
\[\begin{tikzcd}
& \bullet \ar[d, dotted]\ar[dl, dotted, red]& \bullet \ar[d, dotted] \ar[l, tail, red] \ar[dr, dotted, red] & \\
\bullet & \ar[l, dotted, blue]\bullet & \bullet \ar[l, tail, blue] \ar[r, dotted, blue]& \bullet 
\end{tikzcd}\]
in $\cal{W}(\cal{A}_{hG})$.
\end{definition}

\begin{definition}\label{commonsub}
    We say that two DMC-spans: 
   \[\begin{tikzcd}
       \{a_{i}\}_{i \in I} & \{\Phi_j b_{j}\}_{j \in J} \ar[l, dotted, "\{\phi_j\}"']& \{ b_{j}\}_{j \in J} \ar[l, tail, "\{\Phi_j\}"'] \ar[r, dotted, "\{\psi_j\}"] & \{c_{k}\}_{k \in K}\\
      \{a_{i}\}_{i \in I} & \{\Phi'_j b'_{j}\}_{j \in J'} \ar[l, dotted, "\{\phi'_j\}"']& \{ b'_{j}\}_{j \in J'} \ar[l, tail, "\{\Phi'_j\}"'] \ar[r, dotted, "\{\psi'_j\}"] & \{c_{k}\}_{k \in K}
   \end{tikzcd}\]
    \emph{differ by a common subdivision} if there exists a commuting diagram of the form: 
    \[\begin{tikzcd}
      & \{\Phi_j b_{j}\}_{j \in J} \ar[dl, dotted, "\{\phi_j\}"']& \{ b_{j}\}_{j \in J} \ar[l, tail, "\{\Phi_j\}"'] \ar[rd, dotted, "\{\psi_j\}"] & \\
      \{a_{i}\}_{i \in I}  &\{\Xi_l x_l\}_{l \in L} \ar[u, dotted, "\{\gamma_l\}"] \ar[d, dotted, "\{\theta_l\}"']&\{x_l\}_{l \in L} \ar[l, tail, "\{\Xi_l\}"'] \ar[u, dotted, "\{\omega_l\}"] \ar[d, dotted, "\{\lambda_l\}"']& \{c_k\}_{k \in K}\\
    & \{\Phi'_j b'_{j}\}_{j \in J'} \ar[ul, dotted, "\{\phi'_j\}"]& \{ b'_{j}\}_{j \in J'}. \ar[l, tail, "\{\Phi'_j\}"'] \ar[ur, dotted, "\{\psi'_j\}"'] & 
   \end{tikzcd}\]
    
 \end{definition}

\begin{proposition}\label{dmcrep}
    Two DMC-spans representing the same morphism in $\cal{G}(\cal{A})$ have a common subdivision. 
\end{proposition}

\begin{proof}
Given a diagram of two DMC-spans as in \cref{commonsub}  which commute in $\cal{G}(\cal{A})$ we can complete the cospan of covering sub-maps on the right to a square of covering sub-maps by \cref{nomovesquare}: 

\[
\begin{tikzcd}
    & \{\Phi_j b_{j}\}_{j \in J} \ar[dl, dotted, "\{\phi_j\}"']& \{ b_{j}\}_{j \in J} \ar[l, tail, "\{\Phi_j\}"'] \ar[rd, dotted, "\{\psi_j\}"] & \\
      \{a_{i}\}_{i \in I}  & &\{x_l\}_{l \in L}  \ar[u, dotted, "\{\eta_l\}"] \ar[d, dotted, "\{\lambda_l\}"']& \{c_k\}_{k \in K}\\
    & \{\Phi'_j b'_{j}\}_{j \in J'} \ar[ul, dotted, "\{\phi'_j\}"]& \{ b'_{j}\}_{j \in J'}. \ar[l, tail, "\{\Phi'_j\}"'] \ar[ur, dotted, "\{\psi'_j\}"'] & 
\end{tikzcd}
\]
The above diagram commutes in $\cal{G}(\cal{A})$since the outer hexagon commutes and the square on the right commutes. 
\par Now we use \cref{completingsquares} to complete the covering sub-maps followed by moves to squares obtaining a commuting diagram of the form:
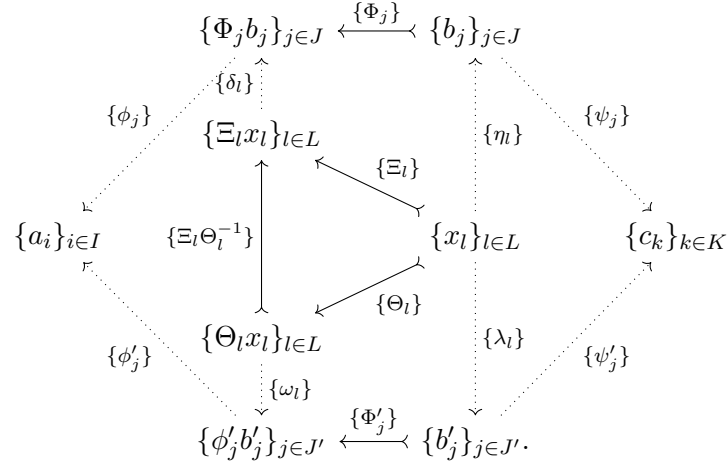
\[\begin{tikzcd}
    & \{\Phi_j b_{j}\}_{j \in J} \ar[ddl, dotted, "\{\phi_j\}"']& \{ b_{j}\}_{j \in J} \ar[l, tail, "\{\Phi_j\}"'] \ar[rdd, dotted, "\{\psi_j\}"] & \\
    &\{\Xi_l x_l\}_{l \in L}  \ar[u, dotted, "\{\delta_l\}"]& &\\
      \{a_{i}\}_{i \in I}  & &\{x_l\}_{l \in L}  \ar[uu, dotted, "\{\eta_l\}"'] \ar[dd, dotted, "\{\lambda_l\}"] \ar[ul, tail, "\{\Xi_l\}"'] \ar[dl, tail, "\{\Theta_l\}"]& \{c_k\}_{k \in K}\\
      &\{\Theta_l x_l\}_{l \in L}  \ar[uu, tail, "\{\Xi_l \Theta^{-1}_l\}"]\ar[d,dotted, "\{\omega_l\}"] &&\\
    & \{\phi'_j b'_{j}\}_{j \in J'} \ar[uul, dotted, "\{\phi'_j\}"]& \{ b'_{j}\}_{j \in J'}. \ar[l, tail, "\{\Phi'_j\}"'] \ar[uur, dotted, "\{\psi'_j\}"'] &
\end{tikzcd}\]
The triangle on the left commutes in $\cal{G}(\cal{A})$: 
\[
(\phi'_{\underline{\omega}(l)}\omega_l,1_G)=(\phi_{\underline{\delta}(l)}\delta_l,\Xi_l \Theta^{-1}_l).
\]
Assuming that all empty covers have been removed, this implies $1_G = \Xi_l \Theta_l^{-1} $ and therefore $\{\Xi_l x_l\}=\{\Theta_l x_l\}$. An updated diagram is displayed below.
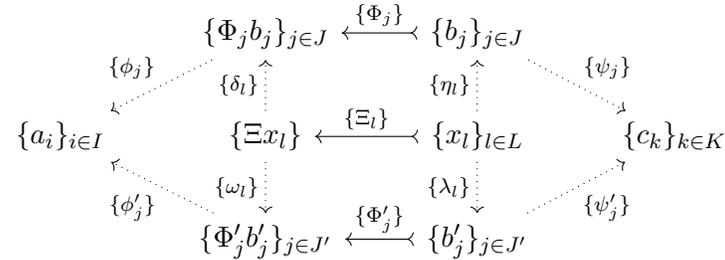
\[
\begin{tikzcd}
    & \{\Phi_j b_{j}\}_{j \in J} \ar[dl, dotted, "\{\phi_j\}"']& \{ b_{j}\}_{j \in J} \ar[l, tail, "\{\Phi_j\}"'] \ar[rd, dotted, "\{\psi_j\}"] & \\
      \{a_{i}\}_{i \in I}  & \{\Xi x_l\}\ar[u, dotted, "\{\delta_l\}"] \ar[d, dotted, "\{\omega_l\}"']&\{x_l\}_{l \in L}  \ar[l, tail, "\{\Xi_l\}"']\ar[u, dotted, "\{\eta_l\}"] \ar[d, dotted, "\{\lambda_l\}"']& \{c_k\}_{k \in K}\\
    & \{\Phi'_j b'_{j}\}_{j \in J'} \ar[ul, dotted, "\{\phi'_j\}"]& \{ b'_{j}\}_{j \in J'} \ar[l, tail, "\{\Phi'_j\}"'] \ar[ur, dotted, "\{\psi'_j\}"'] & 
\end{tikzcd}
\]

\par Now we make a refinement to ensure the triangle on the left commutes in $\cal{W}(\cal{A})$. Since the triangle commutes in $\cal{G}(\cal{A})$ there must a morphism $(\gamma_l,\Gamma_l)_{l \in L'}$which equalizes the arrows in $\cal{W}(\cal{A})$. Using \cref{factoring}, we factor the equalizing arrow into a covering sub map $\{\Gamma_l\}_{l \in L'}$ followed by a a move $\{\gamma_l\}_{l \in L'}$. After this we may complete the resulting cospan by \cref{completingsquares} .
\[\begin{tikzcd}
&\{x_l\}_{l \in L} \ar[d, "\{\Xi_l\}", tail]&\{y_l\}_{l \in L'} \ar[d, "\{\Xi_l\}", tail] \ar[l, dotted]&\\
    \{a_i\}_{i \in I} & \{\Xi_l x_l\}_{l \in L} \ar[l, "\{\phi'_{\underline{\omega}(l)} \circ \omega_l\} "shift left=.75ex] \ar[l, "\{\phi_{\underline{\delta}(l)}\} "',shift right = .75ex ] & \{\Gamma_l y_l\}_{l \in L'} \ar[l, "\{\gamma_l\}", dotted] & \{\Gamma_l y_l\}_{l \in L'} \ar[l, "\{\Gamma_l\}", tail]
\end{tikzcd}\]

Modifying the diagram by putting $\{Y_l\}$ in place of $\{X_l\}$ forces the triangle on the left to commute.
\[
\begin{tikzcd}
    & \{\phi_j b_{j}\}_{j \in J} \ar[dl, dotted, "\{\phi_j\}"']& \{ b_{j}\}_{j \in J} \ar[l, tail, "\{\Phi_j\}"'] \ar[rd, dotted, "\{\psi_j\}"] & \\
      \{a_{i}\}_{i \in I}  &\{\Xi_l y_l\}_{l \in L'}\ar[u, dotted, "\{\delta_l\}"
      ] \ar[d, dotted, "\{\omega_l\}"]&\{y_l\}_{l \in L'}  \ar[l, tail, "\{\Xi_l\}"']\ar[u, dotted, "\{\eta_l\}"] \ar[d, dotted, "\{\lambda_l\}"']& \{c_k\}_{k \in K}\\
    & \{\Phi'_j b'_{j}\}_{j \in J'} \ar[ul, dotted, "\{\phi'_j\}"]& \{ b'_{j}\}_{j \in J'} \ar[l, tail, "\{\Phi'_j\}"'] \ar[ur, dotted, "\{\psi'_j\}"'] & 
\end{tikzcd}
\]
Therefore the given DMC-spans differ by a common subdivision. 
\end{proof}

\begin{section}{The regulator for group completion $K$-theory}
Now we will review the construction of the regulator $\cal{W}(\cal{A}_{hG}) \rightarrow (HA)_{hG}$ defined in \cite{bgmmz} .  Our aim is to extend this map to $\cal{G}(\cal{A}_{hG})$, proving \cref{theorem:regulatorextension}.
\begin{definition}[{\cite[Definition 2.11]{bgmmz}}]\label{def:cat_wedge}
     For $X$ a finite based set and $\cal{A}$ a category we define $X \sma \cal{A}=\cal{A}^{\vee X^\circ}$, where $X^\circ$ is $X$ with the basepoint removed. If $M$ is an abelian group $X \sma M=M^{\oplus X^\circ}$.
\end{definition}
\begin{definition}[{\cite[Definitions 6.1, 6.11 ]{bgmmz}}]\label{regulatorW}
    
 Recall that given a $G$-weak assembler $\cal{A}$, or more generally a $G$-category with covering families, a \emph{measure} on $\cal{A}$ is a $G$-module $A$ and a function $\mu$ satisfying one of the following equivalent conditions: 

\begin{itemize}
\item A $G$-equivariant function $\mu: \ob \cal{A} \rightarrow A$ such that $\mu(P)= \sum_{i \in I} \mu(a_i)$ for every covering family $\{a_i \rightarrow a\}_{i \in I}$ in $\cal{A}$. 
\item A map of $\bb{Z}[G]$-modules $\mu: K_0 (\cal{A}) \rightarrow A$.
\end{itemize}
The \emph{explicit trace} is the map $K_n (\cal{A}_{hG}) \rightarrow H_n(G;A)$ arising from a measure. We call this map the \emph{regulator} or $\mu$-\emph{regulator} if specificity is desired. The regulator obtained from the identity on $K_0(\cal{A})$ is the \emph{universal regulator}. 
\par Let $N_\bullet ^\otimes $ be the two-sided bar construction with the tensor product.  The regulator $R$ is defined in \cite{bgmmz} as a map of simplicial sets,
\[ 
R: N_\bullet \cal{W}(\cal{A}_{hG}) \rightarrow N_\bullet ^\otimes (\bb{Z} , \bb{Z}[G], A) 
\]
defined by sending the $p$-tuple 
\[
 \{a_{0|i}\}_{i \in I_0} \xleftarrow{(\phi_{1|i},\Phi_{1|i})} \{a_{1|i}\}_{i \in I_1} \xleftarrow{(\phi_{2|i}, \Phi_{2|i})} \dots \xleftarrow{(\phi_{p|i}, \Phi_{p|i})} \{a_{p|i}\}_{i \in I_p}
\]
to the sum  
\[ 
\sum_{i \in I_p} \Phi_{1|\underline{\phi_2}\dots \underline{\phi_p} (i)} \otimes \Phi_{2|\underline{\phi_3} \dots \underline{\phi_p} (i)} \otimes \dots \otimes \Phi_{p-1| \underline{\phi_p} (i)} \otimes \Phi_{p|i} \otimes \mu(a_{p|i})
\]
in $C_p (G;A)=\bb{Z}[G]^{\otimes p} \otimes A$.

By \cite[Lemma 6.10]{bgmmz} the regulator generalizes to a map of simplicial sets: 
\[ 
R_X : N_\bullet \cal{W}(X \sma \cal{A}_{hG}) \rightarrow X \otimes N_\bullet ^\otimes (\bb{Z}, \bb{Z}[G], A).
\]
 The geometric realization of this map $|R_X|$ defines a map of special $\Gamma$-spaces, and therefore a map on symmetric spectra: 
\[
K(\cal{A}_{hG}) \xrightarrow{r} |N_{\bullet}^\otimes (\bb{Z}, \bb{Z}[G], A)|\simeq (HA)_{hG}
\]
\end{definition}
\end{section}
\par For our purposes, we will construct a regulator map out of the nerve of the localization of $\cal{W}(\cal{A}_{hG})$ in the case that $\cal{A}_{hG}$ is a weak assembler. In this case, $\cal{A}_{hG}$ is also a weak assembler \cref{prop:htpyorbitwassm}.  
For $\cal{A}$ a weak assembler we define $\cal{G}(\cal{A})=\cal{W}(\cal{A})[\cal{W}(\cal{A})^{-1}]$, the groupoid constructed from $\cal{A}_{hG}$. There is a localization map $L:\cal{W}(\cal{A}_{hG} )\rightarrow \cal{G}(\cal{A}_{hG})$, which in the span model for the completion maps $c \xleftarrow{f} c'$ to the span $c \xleftarrow{f} c' \xrightarrow{\id} c'$.   
\par Our aim is to construct a regulator
\[R^\cal{G}: N_\bullet \cal{G}(\cal{A}_{hG}) \rightarrow N^\otimes _\bullet (\bb{Z}, \bb{Z}[G], A) \]
such that $R^\cal{G}$ extends the regulator $R$ defined in \cref{regulatorW}: $R = R^\cal{G} \circ N_\bullet (L)$. 

\begin{definition}\label{definition:viaduct}
 We call a flag of DMC-spans of length $m$ a \emph{viaduct of length m} if the dissection of the $i$-th DMC-span agrees with the $(i+1)$-st covering sub-map DMC-span.  We write a length $m$ viaduct: 
 \[ \begin{tikzcd}
 \{b_{0|j}\}_{j \in J}   \ar[d, dotted, "\{ \psi_{0|j}\}"]& \{ b_{1|j} \}_{j \in J }\ar[l, tail, "\{\Phi_{1|j}\}"']\ar[d, dotted, "\{\psi_{1|j}\}"] & \ar[l, tail, "\{\Phi_{2|j}\}"'] \dots 
 & \ar[l, tail, "\{\Phi_{m-1|j}\}"'] \{b_{m-1| j}\}_{j \in J}\ar[d, dotted, "\{\psi_{m-1|j}\}"]& \ar[l, tail, "\{\Phi_{m|j}\}"'] \{b_{m | j}\}_{j \in J}\ar[d, dotted, "\{\psi_{m |j}\}"] \\
 \{a_{0|i}\}_{i \in I_0}   &  \{ a_{1|i} \}_{i \in I_1}           &         &   \{a_{m-1| i}\}_{i \in I_{m-1}}         &     \{a_{m | i}\}_{i \in I_m}
 \end{tikzcd} \]
 The vertical dotted arrows are the dissections and covering sub-maps which coincide. 
 \end{definition}
 
\begin{remark}
In the notation used for the viaduct we have $b_{k|j}=\Phi_{k+1|j} \Phi_{k+2|j} \dots \Phi_{m|j} b_{m|j}$. 
\end{remark}
\par  We extend the notion of refinement and common refinement from DMC-spans to viaducts. When no confusion may arise, we will declutter diagrams by omitting the names of morphisms and representing objects by dots. 
 \begin{definition}\label{viaductrefinement}
Given two viaducts of  length $m$: 
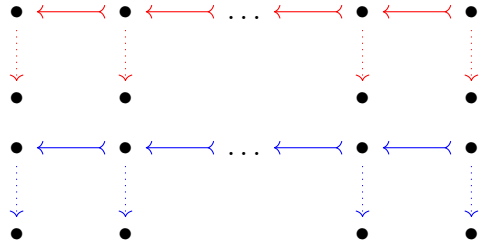
\[ \begin{tikzcd}
    \bullet \ar[d,red, dotted]& \bullet \ar[l,red, tail] \ar[d,red, dotted]& \dots \ar[l,red, tail] & \bullet \ar[l,red,tail] \ar[d,red,dotted]& \bullet\ar[l,red,tail] \ar[d,red,dotted]\\
    \bullet&\bullet&&\bullet&\bullet
\end{tikzcd}\]
\[ \begin{tikzcd}
    \bullet \ar[d, blue, dotted]& \bullet \ar[l,blue,tail] \ar[d,blue,dotted]& \dots \ar[l,blue,tail] & \bullet \ar[l,blue,tail] \ar[d,blue,dotted]& \bullet\ar[l,blue,tail] \ar[d,blue,dotted]\\
    \bullet&\bullet&&\bullet&\bullet
\end{tikzcd}\]
we say that the red viaduct \emph{refines} the blue viaduct if there are black covering maps making the following diagram commute: 
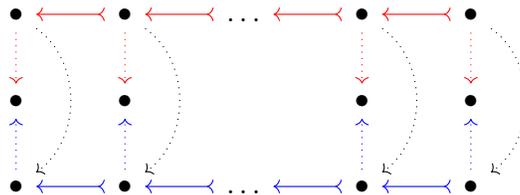
\[\begin{tikzcd}
   \bullet \ar[d, red, dotted] \ar[dd, dotted, bend left = 55]& \bullet \ar[l, red, tail]\ar[d, red, dotted] \ar[dd, dotted, bend left = 55] & \dots \ar[l, red, tail]&\bullet \ar[l, red, tail] \ar[d, red, dotted] \ar[dd, dotted, bend left = 55]& \bullet \ar[l, red, tail] \ar[d, red, dotted] \ar[dd, dotted, bend left = 55]\\
   \bullet & \bullet &  & \bullet & \bullet \\
   \bullet \ar[u, blue, dotted]&\ar[l, blue, tail] \ar[u, blue, dotted] \bullet &\ar[l, blue, tail]\dots& \bullet\ar[l, blue, tail] \ar[u, blue, dotted]& \bullet \ar[l, blue, tail] \ar[u, blue, dotted]
\end{tikzcd}\] 
This requires that the red and blue covering maps share the same codomains. 
 \end{definition}
 \begin{definition}\label{viaductcommonsubdivision}
 We say that the red and blue viaducts: 
 
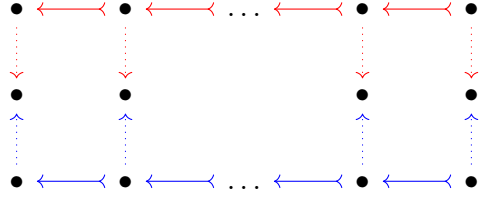
\[ \begin{tikzcd}
     \bullet \ar[d, red, dotted] & \bullet \ar[l, red, tail]\ar[d, red, dotted] & \dots \ar[l, red, tail]&\bullet \ar[l, red, tail] \ar[d, red, dotted] & \bullet \ar[l, red, tail] \ar[d, red, dotted] \\
   \bullet & \bullet &  & \bullet & \bullet \\
   \bullet \ar[u, blue, dotted]&\ar[l, blue, tail] \ar[u, blue, dotted] \bullet &\ar[l, blue, tail]\dots& \bullet\ar[l, blue, tail] \ar[u, blue, dotted]& \bullet \ar[l, blue, tail] \ar[u, blue, dotted]
 \end{tikzcd}\]
 \emph{differ by a common subdivision} if  there are objects and morphisms, represented by black arrows, yielding a commutative  diagram: 
 
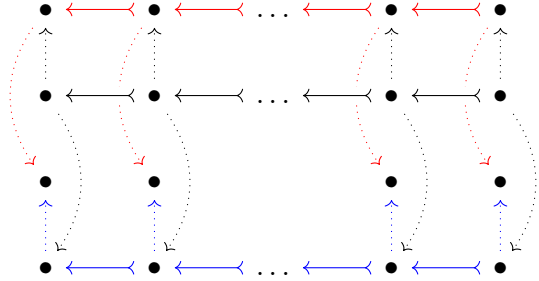
\[\begin{tikzcd}
     \bullet \ar[dd, red, dotted, bend right = 35]& \bullet \ar[dd, red, dotted, bend right = 35] \ar[l, red, tail]& \dots \ar[l, red, tail]& \bullet \ar[l,red,tail]\ar[dd, red, dotted, bend right = 35]& \bullet \ar[dd, red, dotted, bend right = 35] \ar[l,red, tail]\\
     \bullet \ar[dd, dotted, bend left= 35] \ar[u, dotted]& \bullet \ar[l, tail, crossing over] \ar[dd, dotted, bend left= 35]\ar[u, dotted]& \ar[l, tail]\dots & \bullet \ar[dd, dotted, bend left= 35]\ar[u, dotted]\ar[l, tail, crossing over]& \bullet \ar[dd, dotted, bend left= 35]\ar[u, dotted]\ar[l, tail, crossing over]\\
     \bullet & \bullet &  & \bullet & \bullet \\
     \bullet \ar[u, blue, dotted] & \bullet \ar[u, blue, dotted] \ar[l, blue, tail]& \dots \ar[l, blue, tail] & \bullet \ar[u, blue, dotted] \ar[l, blue, tail]& \bullet \ar[u, blue, dotted] \ar[l, blue, tail]
 \end{tikzcd}\] 
 \end{definition}
 
 \begin{lemma}\label{lemma:flagrepex}
 Each flag in $\cal{G}(\cal{A}_{hG})$ is represented by a viaduct.
\end{lemma}

\begin{proof}  First we show that given a viaduct $b$ as in \cref{definition:viaduct}  and a  covering sub-map $\{\xi_l\}:\{x_l\}_{l \in L} \dashrightarrow \{b_{n|j}\}$ of the last tuple in the viaduct, we may construct an equivalent viaduct with $\{x_l\}$ in place of $\{b_{n|j}\}$. Let $\\underline{x}i: L \rightarrow  J$ denote the set map. For $1\leq k \leq n$ and $l \in L$, define 
\begin{align*}
 \Xi_{k|l}&=\Phi_{k|\phi_k \dots \phi_n \underline{\xi} (l)} \\
 x_{n|l}&=X_l \\
x_{k|l} &= \Xi_{k+1|l} \Xi_{k+2|l} \dots \Xi_{n|l} x_{n|l}, k <n.
\end{align*}
Similarly, we can define sub-covers $\{\eta_{k|l}\}: \{x_{k|l}\}_{l \in L} \dashrightarrow \{a_{k|i}\}_{i \in I_k}$ with set maps given by composition, $\eta_k = \psi_k \circ \xi: L \rightarrow I_k$ and $\eta_{k|l}=\psi_{k|\eta_{k} (l)}$. By construction, the viaduct 
 \[ \begin{tikzcd}
 \{x_{0|l}\}_{l \in L}   \ar[d, dotted, "\{ \eta_{0|l}\}"]& \{ x_{1|l} \}_{l \in L }\ar[l, tail, "\{\Xi_{1|l}\}"']\ar[d, dotted, "\{\eta_{1|l}\}"] & \ar[l, tail, "\{\Xi_{2|l}\}"'] \dots 
 & \ar[l, tail, "\{\Xi_{m-1|l}\}"'] \{x_{m-1| l}\}_{l \in L}\ar[d, dotted, "\{\eta_{m-1|l}\}"]& \ar[l, tail, "\{\Xi_{m|l}\}"'] \{x_{m | l}\}_{l \in L}\ar[d, dotted, "\{\eta_{m |l}\}"] \\
 \{a_{0|i}\}_{i \in I_0}   &  \{ a_{1|i} \}_{i \in I_1}           &         &   \{a_{m-1| i}\}_{i \in I_{m-1}}         &     \{a_{m | i}\}_{i \in I_m}
 \end{tikzcd} \]
 is equivalent to the one given. 
 \par Now we give an inductive argument that every flag of length $(m+1)$ is represented by a viaduct. The base case is Proposition 1.6. Now we assume that the statement holds up to $m$. Given a $(m+1)$-flag of morphisms in $\cal{G}(\cal{A})$ we may reduce the first $m$-morphisms into the form of a viaduct.  Using the notation from the statement of this lemma for the first $m$ morphisms and denote the $(m+1)$-st morphism using objects $\{r_k\}_{k \in K}, \{z_l\}_{l \in L}$, and $\{y_s\}_{s \in S}$.
  \[ \begin{tikzcd}
 \{b_{0|j}\}_{j \in J}   \ar[d, dotted, "\{ \psi_{0|j}\}"]& \ar[l,tail, "\{\Phi_{1|j}\}"'] \dots 
 & \ar[l, tail, "\{\Phi_{m|j}\}"'] \{b_{m | j}\}_{j \in J}\ar[d, dotted, "\{\psi_{m |j}\}"] &&&\\
 \{a_{0|i}\}_{i \in I_0}            &               &     \{a_{m | i}\}_{i \in I_m} & \{r_k\}_{k \in K} \ar[l, dotted]& \{z_{k}\}_{k\in K} \ar[l, tail] \ar[r, dotted] & \{y_s\}_{s \in S} 
 \end{tikzcd} \]
 First complete $\{b_{m|j}\}_{j \in J} \dashrightarrow \{a_{m|i}\}_{i \in I_m} \dashleftarrow \{r_k\}_{k \in K}$ via  \cref{nomovesquare}  to a square  using an object $\{x_l\}_{l \in L}$. 
 Then complete the diagram $\{x_l\}_{l \in L} \dashrightarrow \{r_k\}_{k \in K} \leftarrowtail \{z_t\}_{t \in T}$ to a square by an object $\{w_l\}_{l \in L}$ \cref{completingsquares}.  This produces the diagram 
  \[ \begin{tikzcd}
 \{b_{0|j}\}_{j \in J}   \ar[d, dotted, "\{ \psi_{0|j}\}"]& \ar[l,tail, "\{\Phi_{1|j}\}"'] \dots 
 & \ar[l, tail, "\{\Phi_{m|j}\}"'] \{b_{m | j}\}_{j \in J}\ar[d, dotted, "\{\psi_{m |j}\}"] &\{x_l\}_{l \in L} \ar[l, dotted] \ar[d, dotted]&\{w_l\}_{l \in L} \ar[l, tail] \ar[d, dotted]&\\
 \{a_{0|i}\}_{i \in I_0}            &               &     \{a_{m | i}\}_{i \in I_m} & \{r_k\}_{k \in K} \ar[l, dotted]& \{z_{k}\}_{k\in K} \ar[l, tail] \ar[r, dotted] & \{y_s\}_{s \in S} 
 \end{tikzcd} \]
 By our first observation, use the  subcover $\{b_{m|j}\}_{j \in J} \dashleftarrow \{x_l\}_{l \in L}$ to rewrite the diagram
 \[ \begin{tikzcd}
 \{x_{0|l}\}_{l \in L}   \ar[d, dotted, "\{ \eta_{0|l}\}"] & \ar[l, tail, "\{\Xi_{1|l}\}"'] \dots 
 & \ar[l, tail, "\{\Xi_{m|l}\}"'] \{x_{m | l}\}_{l \in L}\ar[d, dotted, "\{\eta_{m |l}\}"] &\{w_l\}_{l \in L} \ar[l, tail] \ar[d, dotted]\\
 \{a_{0|i}\}_{i \in I_0}           &              &     \{a_{m | i}\}_{i \in I_m}& \{z_k\}_{k \in K}
 \end{tikzcd} \]
 producing a viaduct equivalent to the given flag of length $m+1$. 
 \end{proof}
 \begin{lemma}\label{lemma:flagrepunique}
     The representative of \cref{lemma:flagrepex} is unique up to common subdivision.
 \end{lemma}
 
\begin{proof} For uniqueness we induct on the length of the viaduct. The case $n=1$ is proved by \cref{dmcrep}. Now assume the result holds for all viaducts of length $k \leq n$. Suppose that the blue and red viaducts of length $n+1$ represent the black flag of spans. 
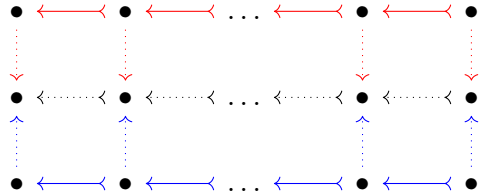
\[\begin{tikzcd}
\bullet \ar[d, dotted, red]& \ar[l, tail, red] \bullet \ar[d, dotted, red]& \ar[l, tail, red]\dots & \ar[l, tail, red]\bullet \ar[d, dotted, red]& \ar[l, tail, red]\bullet \ar[d, dotted, red]\\
 \bullet  & \ar[l, dotted, tail]\bullet & \ar[l, dotted, tail] \dots &  \ar[l, dotted, tail] \bullet & \ar[l, dotted, tail] \bullet   \\
 \bullet \ar[u, dotted, blue]& \ar[l, tail, blue]\bullet \ar[u, dotted, blue]& \ar[l, tail, blue]\dots &\ar[l, tail, blue] \bullet \ar[u, dotted, blue]& \ar[l, tail, blue]\bullet \ar[u, dotted, blue]
\end{tikzcd}
\]
By induction, the first $n$-morphisms of the viaducts differ by a common subdivision and by \cref{dmcrep} the last two DMC-spans have a common subdivision. The arrows defining the common subdivision of the length $n$ viaducts are purple and the common subdivision between the last two DMC-spans is cyan in the following diagram. 
 
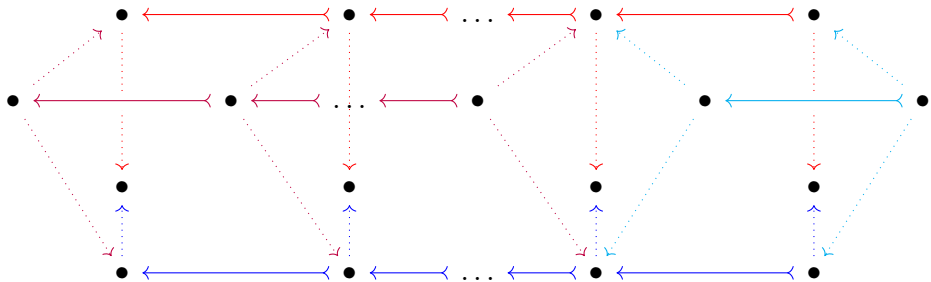
\[\begin{tikzcd}
     & \bullet \ar[dd, dotted, red] && \ar[ll, red, tail]\ar[dd, dotted, red]\bullet &\ar[l, tail, red] \dots&\ar[l, tail, red]  \bullet\ar[dd, red, dotted]&&\ar[ll, tail, red] \bullet\ar[dd, red, dotted]&\\
     \bullet\ar[ur, purple, dotted] \ar[ddr, purple, dotted]&&\ar[ll, tail, purple, crossing over]\bullet \ar[ur, dotted, purple] \ar[ddr, dotted, purple]&\ar[l, tail, purple] \dots& \ar[l, purple, tail] \bullet \ar[ur, purple, dotted] \ar[ddr, purple, dotted]&&\bullet \ar[ul, dotted, cyan] \ar[ddl, dotted, cyan]&&\ar[ll, tail, cyan, crossing over]\bullet \ar[ul, dotted, cyan]\ar[ddl, dotted, cyan]\\
     &\bullet&&\bullet &&\bullet&&\bullet &\\
     &\bullet \ar[u, blue, dotted] &&\ar[ll, blue, tail] \bullet \ar[u, blue, dotted]&\ar[l, blue, tail] \dots&\ar[l, blue, tail]\bullet \ar[u,dotted, blue]&&\ar[ll, tail, blue]\bullet\ar[u, blue, dotted]&
 \end{tikzcd}\] 
 Now we complete the cospans consisting of a cyan dotted arrow and a purple dotted arrow to a commutative square consisiting of submaps by \cref{nomovesquare}. The resulting covering submaps are shown in orange. 
 
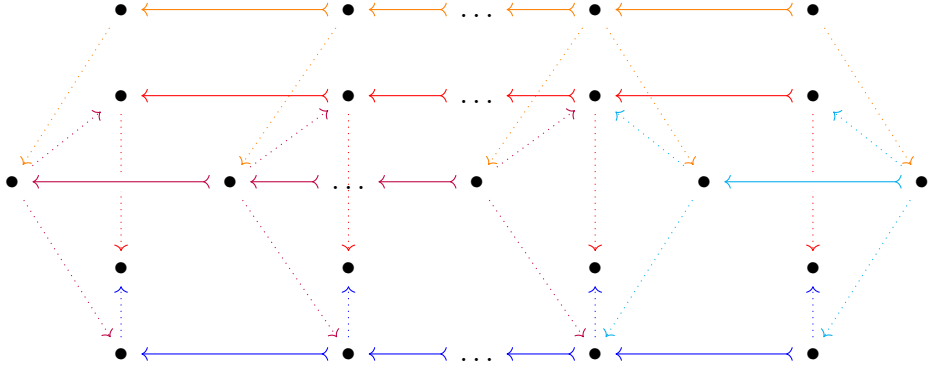
\[\begin{tikzcd}
    &\bullet \ar[ddl, orange, dotted]&&\ar[ll, tail, orange]\bullet\ar[ddl, dotted, orange, crossing over]&\ar[l, tail, orange]\dots&\ar[l, tail, orange]\bullet\ar[ddl, dotted,orange, crossing over]\ar[ddr, dotted, orange, crossing over]&&\ar[ll, orange, tail]\bullet\ar[ddr, orange, dotted]&\\
     & \bullet \ar[dd, dotted, red] && \ar[ll, red, tail]\ar[dd, dotted, red]\bullet &\ar[l, tail, red] \dots&\ar[l, tail, red]  \bullet\ar[dd, red, dotted]&&\ar[ll, tail, red] \bullet\ar[dd, red, dotted]&\\
     \bullet\ar[ur, purple, dotted] \ar[ddr, purple, dotted]&&\ar[ll, tail, purple, crossing over]\bullet \ar[ur, dotted, purple] \ar[ddr, dotted, purple]&\ar[l, tail, purple] \dots& \ar[l, purple, tail] \bullet \ar[ur, purple, dotted] \ar[ddr, purple, dotted]&&\bullet \ar[ul, dotted, cyan] \ar[ddl, dotted, cyan]&&\ar[ll, tail, cyan, crossing over]\bullet \ar[ul, dotted, cyan]\ar[ddl, dotted, cyan]\\
     &\bullet&&\bullet &&\bullet&&\bullet &\\
     &\bullet \ar[u, blue, dotted] &&\ar[ll, blue, tail] \bullet \ar[u, blue, dotted]&\ar[l, blue, tail] \dots&\ar[l, blue, tail]\bullet \ar[u,dotted, blue]&&\ar[ll, tail, blue]\bullet\ar[u, blue, dotted]&
 \end{tikzcd}\] 
 We then complete the resulting cospan consisting of an orange covering submap and a cyan move to a commuting square by \cref{completingsquares}.
\end{proof}

\begin{definition}
    We define the regulator  
\[R^\cal{G}: N_\bullet \cal{G}(\cal{C}_{hG}) \rightarrow N^\otimes _\bullet (\bb{Z}, \bb{Z}[G], A) \]
 by mapping the viaduct displayed in \cref{definition:viaduct} to the sum 
\[ \sum_{j \in J} \Phi_{1|j} \otimes \dots \otimes \Phi_{m|j} \otimes \mu(b_{m|j}). \] The regulator is well defined by \ref{lemma:flagrepunique}. 
\end{definition}
\begin{remark}
    We will use the notation 
    \[ \left[\Phi_{1|j} \mid \dots \mid \Phi_{m|j}\right] \otimes \mu(b_{m|j})=\Phi_{1|j} \otimes \dots \otimes \Phi_{m|j} \otimes \mu(b_{m|j}).\]
\end{remark}
\begin{lemma}\label{regatptsimplicial}
The regulator $R^\cal{G}$ defined above is simplicial.
\end{lemma} 
\begin{proof}
We check that the map $R^\cal{G}$ respects the face and degeneracy maps. 
\end{proof}

\begin{definition}\label{def:reg_pointed}
We define 
\[ R_X ^{\cal{G}} : N_\bullet (\cal{G}(X \sma \cal{C}_{hG}) \rightarrow X \otimes N_\bullet ^\otimes (\bb{Z}, \bb{Z}[G], A)\]
by the same formula as $R^\cal{G}$, with the modification that the term $b_{k|j}$ associated with $x \in X^\circ$ is sent to the summand of $X\otimes N_\bullet ^\otimes (\bb{Z}, \bb{Z}[G], A)$ associated with $x\in X^\circ$. 
\end{definition}
\begin{lemma}\label{lemma:regulator_natural}
    The map $R^\cal{G}_X$ is simplicial and natural in $X$. 
\end{lemma}
\begin{proof}
The proof is the same as \cite[Lemma 6.10]{bgmmz}. Apply \cref{regatptsimplicial} at every point of $X$ and uses the fact that the formula in \ref{regatptsimplicial} is a monoid homomorphism and therefore induces a map on group completion. 
\end{proof}
\begin{theorem}\label{theorem:regulator_extension_proved}
    The regulator $R^{\cal{G}}$ extends $R$. 
\end{theorem}
\begin{proof}
    The regulator $R^{\cal{G}}$ satisfies $R^{\cal{G}} \circ N_\bullet (L) = R$ and is well defined by \cref{lemma:regulator_natural}. 
\end{proof}
\cref{Cor:RatEquiv} follows after a few more steps. 
\begin{definition}\label{def:polytope_group}
  Define the \emph{polytope group of $\cal{A}$} to be $K_0(\cal{A})$. 
\end{definition}
\begin{corollary}[{\cite[Example 7.10]{bgmmz}}]
\label{cor:Cor_D_u_reg}
     Let $\cal{X}$ be the assembler encoding $n$-dimensional no moving spherical, hyperbolic, or Euclidean scissors congruence. Let $G$ be a subgroup of the isometry group. There is a rational equivalence 
      \[ K_* (\cal{X}_{hG}) \rightarrow H_* (G; \Pt(\cal{X}))\]
      induced by the universal regulator.
\end{corollary}
\begin{proof}
    By the main theorem of \cite{Mal24},  the truncation map $K(\cal{X})\rightarrow H(\Pt(\cal{X}))$ is a rational equivalence. Taking homotopy orbits this map defines the abstract trace, which is therefore also a rational equivalence. 
\end{proof}
\begin{remark}
    The only thing distinguishing the above corollary from \cite[Example 7.10]{bgmmz} is that we can now consider the $K$-theory as given by the group completion construction. As we shall see, this is helpful in constructing generators. 
\end{remark}
\section{The regulator for one dimensional geometry}
We are now able to describe the regulator for one dimensional geometry.  
\begin{definition}
    As in \cref{htpyorbiteuc} we use $\cal{E}^1_{e}$ for the no-moving assembler of the line. Taking homotopy orbits by the group of translations $T_1 \cong \bb{R}$ we get the assembler $\cal{E}_T ^1 = (\cal{E}^1)_{hT_1}$ which encodes 1-dimensional translational scissors congruence. 
\end{definition} 
By the Thom spectrum model for $K$-theory we have 

\begin{theorem}
 [{\cite[Theorem 1.5]{Mal24}}] $K_i (\cal{E}^1_T) \cong H_{i+1} (T_1; \on{Pt}(E^1))$ .   
\end{theorem}
 \begin{definition}\label{def:polytope_group_1-dim}
      In the $1$-dimensional case, $\Pt(E^1)$ is the free abelian group generated by intervals $[a,b]$ quotiented by the relation that $[a,b]=[a,x] +[x,b]$ for every $a<x<b$. In other words, $\Pt(E^1)=K_0 (\cal{E}^1_e )$ \cite[Definition 6.7]{klmms24} \cite[Definition 2.11]{Mal24}. There is an action of $T_1$ on $\Pt(E^1)$ defined by $\Psi + [a,b] = [\Psi + a, \Psi + b]$ for $\Psi \in T_1$. 
 \end{definition} 
\begin{definition}
Recall that if $A$ is a torsion free abelian group then the Pontryagin product gives a homomorphism  $H_1 (A)^{\otimes k} \rightarrow H_k (A)$, given by the shuffle formula,  and  induces an isomorphism $H_1(A)^{\sma k} \cong H_k (A)$ \cite[Definition 6.5.11]{weibel1994introduction}. 
Since $H_1(A) \cong A$ we see that $H_k (T_1) \cong T_1^{\sma k }=\bb{R}^{\sma k}$ \cite[Proposition 4.7]{Dup01}.
\end{definition}
\begin{definition}
    We define the volume homomorphism
 \begin{align*} 
 \text{vol} : \on{Pt}(E^1) &\rightarrow \R \\
 [a,b] &\mapsto b-a.
 \end{align*}
  
 \end{definition}

     \begin{proposition}
         The volume regulator agrees with the map on homology that applies $\vol: \Pt(E^1) \rightarrow \bb{R}.$
     \end{proposition}
 \begin{proof}
     
This is a special case of \cite{bgmmz}[Theorem 1.3] which shows that the map
 \[ K(\cal{A})_{hG} \xrightarrow{Asm} K(\cal{A}_{hG}) \xrightarrow{r} (HA)_{hG}\] 
 is the homotopy $G$-orbits of the map 
 \[ K(\cal{A}) \rightarrow HA\]
 defined by the measure $\mu$. Taking $G$-orbits of an Eilenberg-Maclane spectrum defines group homology  $\pi_* (HA)_{hG}\cong H_* (G;A)$. The map $K(\cal{E}^1_e) \rightarrow H (\Pt(E^1))$ is a rational equivalence and after taking homotopy orbits this becomes an equivalence $K(\cal{E}_T ^1) \cong (H( \Pt(E^1)))_{T_1}$. Using this equivalence, the volume regulator is induced by taking homotopy orbits of $H(\Pt(E^1)) \rightarrow H \bb{R}$. 
  \end{proof}
     \begin{proposition}
         $ H_{i}(T_1, \Pt(E^1)) \cong H_{i+1}(T_1; \bb{Z})$. 
     \end{proposition}
     \begin{proof}
         There is a short exact sequence of $T_1$ modules as follows. 
 
 \[ \begin{tikzcd}
  &\left[a,b\right] \ar[r, maps to] &\left[b\right]-\left[a\right] & \\
 0 \ar[r] & \on{Pt}(E^1) \ar[d, "\vol"']\ar[r, "\beta"]& \bb{Z}[\bb{R}] \ar[r, "\epsilon"] & \bb{Z} \ar[r] & 0 \\
 &  \bb{R} & \sum n_r [r] \ar[r , maps to ] & \sum n_r 
 \end{tikzcd}\]
     
 This short exact sequence induces a long exact sequence on homology: 
 
 \[ \dots \rightarrow H_{i+1} (T_1 ; \bb{Z}[\bb{R}] ) \rightarrow H_{i+1} (T_1; \Z) \xrightarrow{\delta} H_i (T_1 ; \Pt(E^1)) \rightarrow H_i (T_1 ; \bb{Z}[\bb{R}]) \rightarrow \dots \]
 
 By Shapiro's Lemma,  
 
 \[ H_* (T_1 ; \bb{Z}[\bb{R}] \otimes\bb{Z}) = H^* (T_1 ; \Hom_\bb{Z} (\bb{Z}[\bb{R}] , \bb{Z})) = \begin{cases} \bb{Z} & *=0 \\ 0 &*\neq 0\end{cases} \]
 and we conclude that $\delta$ is an isomorphism for $i >0$. 
 \par When $i=0$ we check that 
 \[ \beta_*: H_0(T_1, \Pt(E^1)) \rightarrow H_0(T_1 ; \bb{Z}[\bb{R}]) \]
 is the zero map. Since $0$-th homology is coinvariants, we need to  show that 
 \[ \Pt(E^1)_{T_1} \ni [a,b] \mapsto [b] - [a] \equiv 0 \in \bb{Z}[\bb{R}]_{T_1}. \] If $a<b$, there exists $\Psi \in T_1$ such that $\Psi + a = b$. Therefore we see that 
 \[ [b]-[a]=(\Psi + [a]) -[a]\equiv 0 \in \bb{Z}[\bb{R}]_{T_1}. \]
 We conclude that $\beta_*$ defines the zero map on $H_0$ and the snake map $\delta$ is an isomorphism for all $i \geq 0$. 
 
\end{proof}
\begin{remark}
     Note, that $ \vol \circ \beta^{-1} : \im(\beta) \rightarrow \bb{R}$ is given by de-bracketing: $\vol \beta^{-1}[a]=a$. The map given by debracketing is defined on all of $\Z[\bb{R}]$, but is not $T_1$-equivariant. It is only $T_1$-equivariant on the kernel of $\epsilon$.  If it  defined an equivariant map on $\Z[\bb{R}]$, it would make the volume regulator trivial on the higher $K$-groups.
     \end{remark}
\begin{proposition}\label{prop:regularformula}
Identifying $K_n (\cal{E}_T^1) $ with $ H_{n+1}(T_1; \bb{Z})$ by the snake map $\delta$, the regulator $r$ to $H_*(T_1, \bb{R})$ is given by the formula:
  \begin{align*}
H_{n+1}(T_1;\bb{Z}) \cong \bb{R}^{\sma (n+1)}  & \hookrightarrow T_1^{\sma n } \otimes \bb{R} \cong H_{n+1}(T_1,\bb{R}) \\
v_1 \sma \dots \sma v_{n+1} & \mapsto \sum_j (-1)^j (v_1 \sma \dots \sma \widehat{v_j} \sma \dots \sma v_{n+1})\otimes v_j.
\end{align*}
\end{proposition}
\begin{remark}
    The significance of the snake map in this context is that it gives us a way to identify homology classes in $H_k(T_1, \Pt(E^1))$ with homology classes of $H_{k+1}(T_1, \bb{Z})$.  This latter group has a geometric  interpretation. 
\end{remark}
\begin{proof}Following the definition of the snake map we obtain the formula: 
 
\begin{align*}
    v_1 \sma \dots \sma v_{n+1} \in &  H_{n+1}(T_1 ; \bb{Z})\\
    & \uparrow\\
    \sum_{\sigma \in \fS_{n+1}} (-1)^\sigma [v_{\sigma(1)} | \dots | v_{\sigma(n+1)}] \otimes 1 \in & C_{n+1}(T_1; \bb{Z})\\
    & \uparrow \\
    \sum_{\sigma \in \fS_{n+1}} [v_{\sigma(1)}| \dots | v_{\sigma(n+1)}|0] \in & C_{n+1}(T_1; \bb{Z}[\bb{R}])\\
    &  \downarrow d\\
    \sum_{\sigma \in \fS_{n+1}}(-1)^
    \sigma \left(\sum_{j} (-1)^j [v_{\sigma(1)}| \dots | v_{\sigma(j)}+v_{\sigma(j+1)}| \dots | v_{\sigma(n+1)}| 0 ]\right) \in & C_n (T_1 ; \bb{Z}[\bb{R}])\\
    & \downarrow \vol_\bullet\\
    \sum_{\sigma \in \fS_{n+1}} (-1)^\sigma (-1)^{n+1}[v_{\sigma (1)}| \dots | v_{\sigma (n)}] \otimes v_{\sigma (n+1)} \in & C_n (T_1 ; \bb{R})\\
    & \downarrow \\
    \sum_{j} (-1)^j (v_1 \sma \dots \sma \widehat{v}_j \sma \dots \sma v_{n+1}) \otimes v_j \in & H_n(T_1 ; \bb{R})
\end{align*}

The map $\vol_\bullet$ is given by $\vol_\bullet [v_1 \mid \dots \mid v_{n+1}]=[v_1 \mid \dots \mid v_{n}] \otimes v_{n+1}$. Note that  $\vol_\bullet$ is a map of groups, but not chain complexes as it does not respect the differential.  However $\vol_\bullet$ is a map of chain complexes on the kernel of $\epsilon$. In the second to last step we simplify the expression by  using the fact that only the terms with $x \otimes v_{\sigma(n+1)}$ survive since $x \otimes 0=0$ is zero. 
\par We also rearrange the expression 
\[ \sum_{\sigma \in \fS_{n+1}} (-1)^\sigma (-1)^{n+1} [v_{\sigma(1)} \mid \dots \mid v_{\sigma(n)}] \otimes v_{\sigma(n+1)} \]
using the partition $\fS_{n+1} = \bigcup B_j$  where $B_j = \{\sigma \in \fS_{n+1} \mid \sigma (j) = n+1 \}$ so that terms of the form $x \otimes v_j$ are grouped together. This is useful since the map to homology is defined on cycles.
\par The map to $H_n (T_1 , \R)$ is only defined on cycles. It is given by the formula
\[\sum_{\tau \in B_j} (-1)^\tau [ v_{\tau(1)} \mid \dots \mid v_{\tau(n+1)}] \otimes v_j \mapsto (v_1 \sma \dots \sma \hat{v_j} \sma \dots \sma v_{n+1})\otimes v_j \] which concludes the proof. 

\end{proof}

\section{Computing generators of $K_*(\cal{E}_T ^1)$}
We construct generators of  of $K_n (\cal{E}_T ^1)$ by gluing $n!$ simplices together, constructing a subcomplex of $B \cal{G}(\cal{E}_T ^1)$. This turns out to give one of the well known simplicial structures on the torus. We recall some background to describe the construction of the generators. Since $K$-theory of an assembler is a spectrum we need to use some tools from stable homotopy theory. We give a brief presentation of the tools we use. A reference is \cite{adams1978infinite}.
\begin{definition} Let $X$ be a spectrum. Associated to every spectrum is a space $\Omega^\infty (X)$ called the \emph{infinite loop space of X}. The infinite loop space $\Omega^\infty (X)$ is constructed by fibrantly replacing $X$ and then taking the $0$-th space of the resultant fibrant spectrum $R(X)$. This construction is adjoint to the construction of suspension spectrum.
\[\begin{tikzcd}
\hoSpc_*
\arrow[r, "\Sigma^\infty"{name=F}, bend left=25] &
\hoSp_*
\arrow[l, "\Omega^\infty"{name=G}, bend left=25]
%--- Adjunction Symbol
\arrow[phantom, from=F, to=G, "\dashv" rotate=-90]
\end{tikzcd}\]
\par For an infinite loop space $\Omega^\infty (X)$  we define $B^\infty \Omega^\infty (X)$ to be the associated spectrum. The $k$-th level is given by iterating the bar construction $(B^\infty \Omega^\infty (X))_k = B^k \Omega^\infty (X)$ and the bonding maps are equivalences: 
\[ B^k \Omega^\infty (X) \xrightarrow{\sim} \Omega B (B^k \Omega^\infty (X))=\Omega B^{k+1} \Omega^\infty (X)\]
since $B^{k} \Omega^\infty (X)$ is already a group like space. Note that, if $X$ is connective then $X \simeq B^\infty \Omega^\infty (X)$. This material is covered in \cite[Section 1.7]{adams1978infinite} and \cite{spectra_book}.
\end{definition}
\begin{proposition}
    Let $\cal{A}$ be an assembler. A map from an $n$-dimensional torus $\phi: T^n \rightarrow B(\cal{G}(\cal{A}))$ produces an element of $\pi_n ( K(\cal{A}))=K_n (\cal{A})$. 
\end{proposition}
\begin{proof}
Given $\phi$ we can compose with group-completion to obtain a map 
\[ \tilde{\phi} : T^n \rightarrow \Omega B^{\amalg}(B(\cal{G} (\cal{A}))\cong \Omega^\infty K(\cal A).\] 
Since we are mapping into an infinite loop space $\tilde{\phi}$ is equivalent to a map 
\[ \Sigma^\infty_+ T^n \rightarrow B^\infty \Omega^\infty K(\cal{A}) \simeq K(\cal{A}).\]
Finally, observe that there is a map $\Sigma^n \bb{S}= \bb{S}^n \rightarrow \Sigma^\infty _+ T^n $ since a torus is stably equivalent to a wedge of spheres $\bigvee_{k=0}^n \binom{n}{k} (S^k) $, where $\binom{n}{k} (S^k)$ is the wedge of $\binom{n}{k}$ copies of $S^k$. The map defined above maps into the top dimensional sphere $\binom{n}{n} (S^n)$ in the wedge.  
We can think of this as stable version of the fundamental class of the torus.
\par By composition we obtain a map 
\[ \bb{S}^n \rightarrow K(\cal{A})\]
which represents an element of $K_n (\cal{A}).$
\end{proof}

\begin{construction}\label{complexrecipe}
We construct a sub-complex $C_X$ of $B(\cal{G} (\cal{E}^1_T))$ which gives rise to a generator of $K_*(\cal{E}^1_T)$.
\end{construction}

     Let $X'=\{\Phi_0, \dots , \Phi_{n}\}$ be a set of positive real numbers.  Let $X=X' \setminus \{\Phi_{0}\}$ and $\Phi=\Phi_0 + \dots + \Phi_{n}$.
    We define a group homomorphism 
    \[\bb{Z}^{\oplus X} \rightarrow \aut_{\cal{E}_T^1}[0,\Phi]\] by mapping $\Phi_j$ to the the interval exchange transformation rotating the interval $[0,\Phi]$ clockwise by $\Phi_j$. In our model, this corresponds to the DMC-span:
    \[
    \begin{tikzcd}
        \{(\Phi_j -\Phi)+\left[\Phi- \Phi_j,\Phi \right], \Phi_j +\left[0,\Phi- \Phi_j \right]\}\ar[dr, dotted] & &\ar[ll, tail] \{\left[0,\Phi-\Phi_j \right],\left[\Phi-\Phi_j,\Phi \right]\} \ar[dl, dotted]\\
        &\left[0,\Phi \right].&
    \end{tikzcd}
    \]
    This is well defined since circle rotations commute.
    \par Taking classifying spaces we obtain a map 
    \[ 
    B(\bb{Z}^{\oplus X}) \rightarrow B(\aut_{\cal{E}_T^1}[0,\Phi]) \rightarrow B(\cal{G}(\cal{E}_T ^1)) 
    \]
  to the classifying space of the scissors congruence groupoid. The image of this map is a complex $C_X$ which is equivalent to the $n$-dimensional torus with its usual simplicial structure.

\begin{proposition}\label{prop:torusrepresentsgenerator}
    The complex constructed in \cref{complexrecipe}  represents the element $\Phi_0 \sma \Phi_1 \sma \dots \sma \Phi_n \in K_{n}(\cal{E}_T^1)$
\end{proposition}
\begin{proof}
    Since the volume regulator $v:K_* (\cal{E}_T^1) \rightarrow H_*(T_1 ; \bb{R})$  is injective, it suffices to show that the regulator applied to $\Phi_0 \sma \dots \sma \Phi_n$ agrees with the homology class represented by $C_X$. By \cref{prop:regularformula}  we have 
    \[ v(\Phi_0 \sma \dots \sma \Phi_n)=\sum_{j=0}^n (-1)^j (\Phi_0 \sma \dots \sma \widehat{\Phi}_j \sma \dots \sma \Phi_n) \otimes \Phi_j.\]
    
    \par We claim that the fundamental class of $C_X$ corresponds to the generator 
    \[ \Phi_0 \sma \dots \sma \Phi_n  \in K_n(\cal{E}_T ^1).\]
     Let 
    \[ x_k= \left [\sum_{j <k} \Phi_j, \sum_{j \leq k}  \Phi_j \right] \]
   so that we have a decomposition of the interval
    \[ [0,\Phi] = x_0 \sqcup \dots \sqcup x_{n+1} \]
    into subintervals $x_i$. Define
    \begin{align*}
        \Xi_{k|j}&= \begin{cases}
            \Phi_j - \Phi & k=j \\
            \Phi_j & k \neq j
        \end{cases}\\
        x_{n|j}&=x_j.
    \end{align*}
    In this notation, the viaduct 
    \[
    \begin{tikzcd}
    \{x_{0|j}  \} \ar[d,dotted ]& \ar[l, tail,"\{\Xi_{1|j}\}"'] \dots & \ar[l, tail, "\{\Xi_{n-1|j}\}"'] \{x_{n-1|j}\} \ar[d, dotted] & \ar[l, tail, "\{\Xi_{n|j}\}"'] \{x_{n|j}\} \ar[d, dotted]\\
   \left[ 0,\Phi\right] &&\left[ 0,\Phi\right]&\left[ 0,\Phi\right]
    \end{tikzcd}
     \]
encodes the flag of automorphisms consisting of clockwise rotations by $\Phi_n, \Phi_{n-1} , \dots$ and ultimately $\Phi_0$. This is one of the $n!$ simplices which make up $C_X$. Applying the regulator induced by volume to this viaduct yields 
\[ [\Phi_1 | \dots | \Phi_n] \otimes \Phi_0 +\sum_{1 \leq j \leq n} [\Phi_1 | \dots | \Phi_j - \Phi | \dots | \Phi_n] \otimes \Phi_j \in C_n (T_1 ; \bb{R} )\]
where $\Phi_j - \Phi$ is in the$j$-th  position. 
\par The other simplices making up $C_X$ are given by permuting the $x_i$ by $\sigma \in  \fS_{n} $. Applying the regulator to all of these simplices we get 
\begin{align*}
    v(C_X)&=\sum_{\sigma \in \fS_n}(-1)^\sigma \left([\Phi_{\sigma(1)} | \dots | \Phi_{\sigma(n)}] \otimes \Phi_0 +\sum_{1 \leq j \leq n} [\Phi_{\sigma(1)} | \dots | \Phi_{\sigma(j)} - \Phi | \dots | \Phi_{\sigma(n)}] \otimes \Phi_{\sigma(j)} \right).\\
\end{align*}
To simplify notation we set $A_{\sigma} = [\Phi_{\sigma(1)} | \dots | \Phi_{\sigma(n)}] \otimes \Phi_0$. 
Note also that $\sum(-1)^\sigma A_{\sigma} $ represents the homology class  $(\Phi_1 \sma \dots \sma \Phi_n) \otimes \Phi_0$. 
We will simplify the sum to show that it represents the homology class $\sum_j (-1)^j (\Phi_0 \sma \dots \sma \widehat{\Phi}_j \sma \dots \sma \Phi_n) \otimes \Phi_j$. As we have already observed, $\sum_\sigma A_{\sigma}$ represents the first term in this sum. Since $\Phi_j - \Phi = \sum_{i \neq j} \Phi_i$ we use multilinearity to obtain the sum
\begin{align*}
   v(C_X)=&\sum_{\sigma \in \fS_n}(-1)^\sigma ( A_{\sigma} + \sum_{1 \leq j \leq n} (-[\Phi_{\sigma(1)}| \dots |\Phi_0| \dots |\Phi_{\sigma(n)}] \otimes \Phi_{\sigma(j)} \\&- 
    \sum_{k \neq j} [\Phi_{\sigma(1)} | \dots |\Phi_{\sigma(k)} | \dots | \Phi_{\sigma(k)}| \dots | \Phi_{\sigma(n)}] \otimes \Phi_{\sigma(j)} )).
\end{align*}
    Now we observe that the 
\[ 
    \sum_{\sigma \in \fS_n } (-1)^\sigma (\sum_{k \neq j} [\Phi_{\sigma(1)} | \dots |\Phi_{\sigma(k)} | \dots | \Phi_{\sigma(k)}| \dots | \Phi_{\sigma(n)}] \otimes \Phi_{\sigma(j)})=0.
 \]
 Note that $\Phi_{\sigma(k)}$ occurs in the $j$-th position and the $k$-th position. Composing $\sigma$ with the transposition $\tau=(ij)$ changes the signature $-(-1)^\sigma = (-1)^{\sigma \tau}$ so that the terms associated with $\sigma$ and $\sigma \tau $ cancel. Therefore the sum simplifies to 
 \begin{align*}
     \sum_{\sigma \in \fS_n} (-1)^\sigma A_{\sigma} + \sum_{1 \leq k \leq n} \sum_{\sigma \in \fS_n} (-1)^{\sigma + 1} [\Phi_{\sigma(1)}| \dots |\Phi_0 | \dots | \Phi_{\sigma(n)}] \otimes \Phi_k .
 \end{align*}
 In the term on the right $\Phi_0$ appears in the $\sigma^{-1}(k)$-th place. Fixing $k$, the term on the right is a chain representing 
 \[ (-1)^{k-1} (\Phi_0 \sma \dots \sma \widehat{\Phi}_k \sma \dots \sma \Phi_n)\otimes \Phi_k.\] 
 Therefore $r(C_X)$ represents 
 \[\sum_{j=0}^n (\Phi_0 \sma \dots \sma \widehat{\Phi}_j \sma \dots \sma \Phi_n) \otimes \Phi_j.\]
  
\end{proof}
This concludes the proof of \cref{Kgen}. 
\begin{example}
    The element $x\sma y \sma z \in K_2 (\cal{E}_T^1)$ is realized by the torus displayed below. 
   
    \[\begin{tikzcd}
        xyz  \ar[d, "\rho_x"', blue]\ar[r, "\rho_{x+y}", purple]& \begin{matrix}zxy\\xzy\end{matrix} \ar[d, "\rho_x", blue] \\
        \begin{matrix}  
            yzx\\yxz
        \end{matrix} \ar[r, "\rho_{x+y}", purple] \ar[ur, "\rho_y", red]& xzy
    \end{tikzcd}\]
    Abusing notation, we use $x,y,z$ for intervals and also the lengths of those intervals. The string $xyz$ corresponds to the intervals arranged in that order. We use $\rho_x$ as shorthand for rotation clockwise by $x$. 
    
\end{example}

\begin{definition}\label{IETring}
    We define a ring structure on the homology of $IET=\aut_{\cal{E}_T^1}[0,1]$. There is a map 
    \[\aut_{\cal{E}_T^1} [0,1] \times \aut_{\cal{E}_T^1}[0,1] \rightarrow \aut_{\cal{E}_T^1}[0,2]  \]
    defined by stacking automorphisms. If $\Phi$ and $\Psi$ are automorphisms on $[0,1]$ then we define an automorphism on $[0,2]$ which applies $\Phi$ on $[0,1]$ and $\Psi$ on $[1,2]$, where this interval is identified with $[0,1]$ by a translation. Recall that by \cite[Theorem 3.17]{klmms24} there is a canonical isomorphism on homology $H_*( \aut_{\cal{E}_T^1}[0,\Phi] )\cong H_* (\aut_{\cal{E}_T^1}[0,\Psi])$ in all degrees. 
    Therefore we have a multiplication on the homology of $IET$ defined by 
    \[ H_*(\aut_{\cal{E}_T^1}[0,1]) \times H_* (\aut_{\cal{E}_T^1}[0,1]) \rightarrow H_* (\aut_{\cal{E}_T^1}[0,2]) \cong H_* (\aut_{\cal{E}_T^1}[0,1])\]
\end{definition}
Having defined multiplication on $H_* (IET)$ we are now able to prove \cref{cor:IETgen}
\begin{corollary}\label{cor:IETgen_proved}
The subcomplexes $C_X$ generate the homology of $H_*(IET)$ as a ring. 
\end{corollary}
\begin{proof}
    Recall that $H_* (\Omega_0 ^\infty K(\cal{E}_T ^1) )\cong H_* (IET) $ by \cite[Corollary 4.5]{klmms24}. Since the $K$-groups are rational, therefore the homology of $IET$ is the free graded algebra on the $K$-groups of $\cal{E}_T ^1$. We conclude that the generators of the $K$-groups from \cref{complexrecipe} in generate $H_*(IET)$.
\end{proof}
\section{Rectangle exchange transformations and translational scissors congruence}
Using the results of \cite{klmms24} we may extend our result concerning a generating set of $H_*(IET)$ to the case of rectangle exchange groups. We give an outline of how this works. 
\begin{definition}\label{rectangleexchangeassembler}
    \cite[Definition 6.1,6.29]{klmms24}
    There is an assembler which we denote $\cal{R}^n$ where the the objects are non-degenerate polytopes in $\bb{R}^n$ such that every face is orthogonal to a  standard basis vector. Morphisms are inclusions. The group of translations $G=T_n$ acts on this assembler and we call the resulting homotopy orbit assembler  $\cal{R}^n_T=(\cal{R}^n )_{hT_n}$ the \emph{rectangle exchange assembler}. 
\end{definition}

\begin{definition}\label{rectangleexchangegroup}
 We define the rectangle exchange group $Rec_n = \aut_{\mathcal{R}^n _T}([0,1]^n)$ \cite[Section 2]{cornulier2022groupsrectangleexchangetransformations}. 
\end{definition}

 \par By \cite[Theorem 6.16]{klmms24} the $K$-theory spectrum of the rectangle exchange assembler is equivalent to a smash product of the $K$-theory of 1-dimensional translational scissors congruence $K(\cal{R}_T^n) \simeq K(\cal{E}_T^1)^{\sma n}$. The calculation stated in the following proposition is a consequence. 
\begin{proposition}[{\cite[Lemma 6.30, Corollary 6.31]{klmms24}}]\label{summarizedtheoremsRect}
 \begin{align*}
     H_* (Rec_n)&\cong H_*(\Omega_0^\infty K(\cal{R}^n_T)))   \\
     K_*(\cal{R}^n_T) & \cong \bigotimes_{i=1}^n \left(\bigoplus_{m \geq 0} \Lambda_{\bb{Q}}^{m+1} (\bb{R})[m]\right) \\
     H_* (Rec_n) &\cong \Lambda^* \left(\bigotimes_{i=1}^n \left(\bigoplus_{m \geq 0} \Lambda_{\bb{Q}}^{m+1} (\bb{R})[m]\right)\right)
    \end{align*}
\end{proposition}
\begin{corollary}
    Taking Pontryagin products of the generators of $H_*(IET)$ produces a generating set for $H_*(Rec_n)$. 
\end{corollary}

\begin{proof}
  
   Note that $IET^{\times n}$ acts component wise on $[0,1]^n$ by scissors automorphisms. This defines a group homomorphism $IET^n \rightarrow Rec_n$. Therefore we have a map $B(IET)^{\times n} \rightarrow B(Rec_n)$. By \cref{summarizedtheoremsRect} and \cref{complexrecipe} the products of generators of $IET$ generate $Rec_n$ along this map. 
\end{proof}

	% bib
	\bibliographystyle{amsalpha}
	\bibliography{Regulator}
	
\end{document}